\documentclass{amsart}
\usepackage[utf8]{inputenc}
\usepackage[backend=bibtex]{biblatex}
\usepackage{float}
\usepackage{amsmath,amssymb,amsthm} 
\usepackage{graphicx}
\usepackage{caption}
\usepackage{enumerate}
\usepackage{subcaption}
\usepackage{comment}
\usepackage[shortlabels]{enumitem}
\usepackage{relsize}
\usepackage{MnSymbol}
\usepackage{thmtools} 
\usepackage{thm-restate}
\usepackage{anyfontsize}
\usepackage{multicol}
\usepackage{tikz-cd}
\usepackage{biblatex}
\usepackage{shuffle}
\usepackage{hyperref}
\usepackage{xcolor} 
\usepackage{tikz} 
\usepackage{mathdots} 
\usepackage[normalem]{ulem}

\hypersetup{
  colorlinks = true, 
  urlcolor = blue, 
  linkcolor = blue, 
  citecolor = red 
}

\newdimen\R 
\newcommand{\Z}{\mathbb{Z}}

\newcommand{\psitor}{\Psi^{[Q]}_{\rm{tor}}(\mathbf{x})}
\newcommand{\torext}{\mathcal{L}_{\rm{tor}}([Q])}

\theoremstyle{plain}
\newtheorem{theorem}{Theorem}[section]
\newtheorem{definition}[theorem]{Definition}

\newtheorem{lemma}[theorem]{Lemma}
\newtheorem{prop}[theorem]{Proposition}
\newtheorem{cor}[theorem]{Corollary}
\newtheorem{example}[theorem]{Example}
\newtheorem{remark}[theorem]{Remark}
\usepackage{biblatex}

\title{A Toric Analogue for Greene's Rational Function of a Poset}
\date{}
\author{Elise Catania}

\begin{document}

\maketitle
\begin{abstract}
Given a finite poset, Greene introduced a rational function obtained by summing certain rational functions over the linear extensions of the poset. 
This function has interesting interpretations, and for certain families of posets, it simplifies surprisingly. 
In particular, Greene evaluated this rational function for strongly planar posets in his work on the Murnaghan--Nakayama formula.

In 2012, Develin, Macauley, and Reiner introduced toric posets, which  combinatorially are equivalence classes  of posets (or rather acyclic quivers)  under the operation of flipping maximum elements into minimum elements and vice versa. In this work, we introduce a toric analogue of Greene's rational function for toric posets, and study its properties. In addition, we use toric posets to show that the Kleiss--Kuijf relations, which appear in scattering amplitudes, are equivalent to a specific instance of Greene's evaluation of his rational function for strongly planar posets. Also in this work, we give an algorithm for finding the set of toric total extensions of a toric poset. 
\end{abstract}

\section{Introduction} \label{sec:intro}

Let $P$ be a poset on $[n] = \{1, 2, \ldots, n\}$ and $\mathbf{x}=(x_1,x_2, \ldots, x_n)$ be a tuple of indeterminates. A linear extension $w=(w_1,\ldots,w_n)$ of $P$ is a total order
$w_1 < w_2 < \cdots < w_n$ on its elements that extends the partial order; that is, if $w_i <_P w_j$, then $i <_\Z j$. Let $\mathcal{L}(P)$ denote the set of linear extensions of $P$. In 1992, Curtis Greene introduced the following rational function\footnote{One should also note that it is {\bf not} the famous {\it Greene-Kleitman invariant} \cite{greenekleitman} of a poset.} in order to give a combinatorial proof of the well-known Murnaghan--Nakayama formula  \cite{Greene}: 
\begin{equation}\label{Greene's function}
\Psi^P(\mathbf{x}) = \sum_{w \in \mathcal{L}(P)} \frac{1}{(x_{w_1}-x_{w_2})(x_{w_2}-x_{w_3}) \cdots (x_{w_{n-1}}-x_{w_n})}.
\end{equation}
Note that permuting the labels $\{1,2,\ldots,n\}$ of the elements of $P$ will
only permute the variables in the
rational function $\Psi^P(\mathbf{x})$ without significantly changing its ``shape"
e.g. numerator and denominator degrees in simplest form, factorizations thereof, etc.
Part of the mathematical beauty in $\Psi^P(\mathbf{x})$ is that for certain families of posets, $\Psi^P(\mathbf{x})$ simplifies surprisingly. 
Here are two examples.

\begin{table}[h]
\centering
\begin{tabular}{|c|c|c|} 
\hline
{} & $P_1$ & $P_2$ \\
\hline 
{} & \begin{tikzpicture}[scale = 0.95]
\node(1) at (-0.5,0){1};
\node(2) at (-1,1){2};
\node(3) at (0,1){3};
\node(4) at (-0.5,2){4};
\node(5) at (1.5,0){5};
\node(6) at (1,1){6};
\node(7) at (2,1){7};
\draw[line width = 0.25 mm,->](1) -- (2);
\draw[line width = 0.25 mm,->](1) -- (3);
\draw[line width = 0.25 mm,->](2) -- (4);
\draw[line width = 0.25 mm,->](3) -- (4);
\draw[line width = 0.25 mm,->](5) -- (6);
\draw[line width = 0.25 mm,->](5) -- (7);
\end{tikzpicture} &
\begin{tikzpicture}[scale = 0.95]
\node(1) at (-1,0) {2};
\node(2) at (0,0) {1};
\node(3) at (-1.5,1) {3};
\node(4) at (-0.5,1) {4};
\node(5) at (0.5,1) {5};
\node(6) at (0,2) {6};
\draw[line width = 0.25 mm,->](1) -- (3);
\draw[line width = 0.25 mm,->](1) -- (4);
\draw[line width = 0.25 mm,->](2) -- (4);
\draw[line width = 0.25 mm,->](2) -- (5);
\draw[line width = 0.25 mm,->](4) -- (6);
\draw[line width = 0.25 mm,->](5) -- (6);
\end{tikzpicture} \\
\hline 

&&\\

$\Psi_{P_i}(\mathbf{x})$ & $0$ & 
\huge{$\frac{x_1-x_6}{(x_2-x_3)(x_2-x_4)(x_1-x_4)(x_1-x_5)(x_4-x_6)(x_5-x_6)}$} \\

&&\\
\hline 
\end{tabular}
\caption{}
\label{table:1}
\end{table}

 \noindent 
 The two above examples illustrate one of Greene's main results on strongly planar posets.
 A \textit{planar} poset is one whose associated Hasse diagram $H(P)$ is a planar graph. A poset $P$ is \textit{strongly planar} if its Hasse diagram $H(P)$ may be order-embedded in
 $\mathbb{R} \times \mathbb{R}$ without edge crossings, even when an extra minimum element  $\hat{0}$ and maximum element $\hat{1}$ are added to $P$. An example of a poset whose Hasse diagram is a planar graph, but is not a strongly planar poset is the following bow tie poset. 
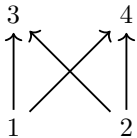
\begin{center}
\begin{tikzpicture}[scale = 0.75]
\node(1) at (-1,0){1};
\node(2) at (1,0){2};
\node(3) at (-1,2){3};
\node(4) at (1,2){4};
\draw[line width = 0.25 mm, ->] (1)  -- (3);
\draw[line width = 0.25 mm,->](2) -- (3);
\draw[line width = 0.25 mm,->] (1)  -- (4);
\draw[line width = 0.25 mm,->](2) -- (4);
\end{tikzpicture}
\end{center}

 One of the characteristics of a strongly planar poset is that its Hasse diagram can be drawn in the plane such that its edges bound regions of the plane that have disjoint interiors from each other; 
 we will call this set of regions $\Delta$. Each bounded region $\delta \in \Delta$ has a unique minimum element $\rm{min(\delta)}$ and a unique maximum element $\rm{max(\delta)}$. We say that a poset $P$ is \textit{connected} if $H(P)$ is a connected graph. Otherwise, the poset is \textit{disconnected}. In \cite{Greene}, Greene proved for a strongly planar poset that $\Psi^P(\mathbf{x})$ vanishes if $H(P)$ is disconnected, and otherwise
\begin{equation} 
\Psi^P(\mathbf{x}) = \frac{\prod_{\delta \in \Delta} (x_{\rm{min(\delta)}}-x_{\rm{max(\delta)}})}{\prod_{i \lessdot_{P} j} (x_i-x_j)}. \label{stronglyplanar}
\end{equation}

In Table \ref{table:1}, the poset $P_2$ is a connected, strongly planar poset and in $H(P_2)$, there is exactly one bounded region $\delta$, with $\max(\delta)=6$ and $\min(\delta)=1$.

In \cite{Cones}, Boussicault, F\'eray, Lascoux, and Reiner interpreted $\Psi^P(\mathbf{x})$ geometrically and algebraically, extending Greene's results. E.g., they showed for all posets $P$ that $H(P)$ is disconnected if and only if $\Psi^P(\mathbf{x})=0$.

 In 2012, Develin, Macauley, and Reiner introduced \textit{toric posets} \cite{toric} (also seen in \cite{macauley2016morphisms}). Geometrically, a toric poset corresponds to a toric chamber in the complement of a graphic toric hyperplane arrangement. This is similar to how a poset corresponds to a chamber in the complement of a graphic hyperplane arrangement; see \cite{toric, Zaslavsky, postnikovfaces, stanleyacyclic} and Section~\ref{sec:posets and hyperplanes}. Combinatorially, a toric poset is an equivalence class $[Q]$ of acyclic quivers that are equivalent under the relation of flipping a sink vertex to a source vertex and vice versa.

\begin{example}\label{Ex: toric poset example}
\rm
Let us consider the following toric poset $[Q]$:

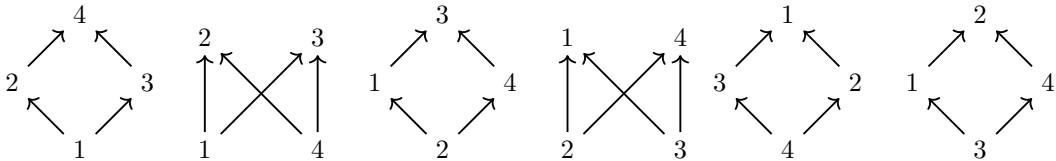
\begin{center}
\begin{tabular}{cccccc}
 \begin{tikzpicture}[scale=0.9]
\node(1) at (0,0) {$1$};
\node(2) at (-1,1) {$2$};
\node(3) at (1,1) {$3$};
\node(4) at (0,2){$4$};
\draw[line width = 0.25 mm, ->] (1) -- (2);
\draw[line width = 0.25 mm, ->] (2) -- (4);
\draw[line width = 0.25 mm, ->] (1) -- (3);
\draw[line width = 0.25 mm, ->] (3) -- (4);
\end{tikzpicture}
& 
\begin{tikzpicture}[scale=0.75]
\node(1) at (-1,0) {$1$};
\node(4) at (1,0) {$4$};
\node(3) at (1,2) {$3$};
\node(2) at (-1,2){$2$};
\draw[line width = 0.25 mm, ->] (1) -- (3);
\draw[line width = 0.25 mm, ->] (1) -- (2);
\draw[line width = 0.25 mm, ->] (4) -- (3);
\draw[line width = 0.25 mm, ->] (4) -- (2);
\end{tikzpicture}

&
\begin{tikzpicture}[scale=0.9]
\node(2) at (0,0) {$2$};
\node(4) at (1,1) {$4$};
\node(1) at (-1,1) {$1$};
\node(3) at (0,2){$3$};
\draw[line width = 0.25 mm, ->] (2) -- (1);
\draw[line width = 0.25 mm, ->] (2) -- (4);
\draw[line width = 0.25 mm, ->] (1) -- (3);
\draw[line width = 0.25 mm, ->] (4) -- (3);
\end{tikzpicture} 

&
\begin{tikzpicture}[scale=0.75]
\node(3) at (1,0) {$3$};
\node(2) at (-1,0) {$2$};
\node(1) at (-1,2) {$1$};
\node(4) at (1,2){$4$};
\draw[line width = 0.25 mm, ->] (2) -- (1);
\draw[line width = 0.25 mm, ->] (2) -- (4);
\draw[line width = 0.25 mm, ->] (3) -- (1);
\draw[line width = 0.25 mm, ->] (3) -- (4);
\end{tikzpicture}

\begin{tikzpicture}[scale=0.9]
\node(4) at (0,0) {$4$};
\node(3) at (1,1) {$3$};
\node(2) at (-1,1) {$2$};
\node(1) at (0,2){$1$};
\draw[line width = 0.25 mm, ->] (2) -- (1);
\draw[line width = 0.25 mm, ->] (3) -- (1);
\draw[line width = 0.25 mm, ->] (4) -- (2);
\draw[line width = 0.25 mm, ->] (4) -- (3);
\end{tikzpicture} 

&
\begin{tikzpicture}[scale=0.9]
\node(3) at (0,0) {$3$};
\node(4) at (1,1) {$4$};
\node(1) at (-1,1) {$1$};
\node(2) at (0,2){$2$};
\draw[line width = 0.25 mm, ->] (1) -- (2);
\draw[line width = 0.25 mm, ->] (4) -- (2);
\draw[line width = 0.25 mm, ->] (3) -- (4);
\draw[line width = 0.25 mm, ->] (3) -- (1);
\end{tikzpicture}
\end{tabular}
\end{center}
One can check that any two representatives $Q_1, Q_2 \in [Q]$ differ by a sequence of source to sink (or sink to source) flips.
\end{example}

 This flip operation has been well-studied and appears widely in different contexts \cite{bernstein, chen, eriksson, macauley2009posets, mosesian, pretzel, Propp, speyer}. In fact, these equivalence classes are subsets of the mutation class of a quiver used in cluster algebras \cite{fomin2}. 
 
 Toric posets can be thought of informally as a cyclic type of poset. Other examples of posets that are cyclic in nature, but are distinct from toric posets, are \textit{partial cyclic orders} \cite{megiddo1976partial} and \textit{affine posets} \cite{galashin2021poset}. Partial cyclic orders will arise in the discussion in Section \ref{sec: relationship to tricolored subdivisions}.
 
 Just as a permutation $(w_1,w_2,\ldots,w_n)$ of $[n]$ may be thought of as a total order $w_1<w_2<\cdots<w_n$ or an acyclic orientation of the complete graph on $[n]$, a {\it toric total order} is the cyclic equivalence class
 $[(w_1,w_2,\ldots,w_n)]$ under rotation $(w_1,w_2,\ldots,w_n) \mapsto (w_2,w_3,\ldots,w_n,w_1)$, or the special case of a toric poset $[Q]$ for an acyclic quiver whose underlying undirected graph is complete. And just as every poset $P$ on $[n]$ has its associated set $\mathcal{L}(P)$ of linear extensions $w$,
 Section \ref{sec: toric posets} will associate to every toric poset $[Q]$ a collection of toric total orders $[w]$ called its set of \textit{toric total extensions}, denoted $\torext$.

 \begin{example} \label{ex: Calculating toric total extensions} \rm The toric poset $[Q]$ from Example \ref{Ex: toric poset example} has four toric total extensions, 
  $$
  \torext=\{[(1,2,3,4)],[(1,3,2,4)],[(1,4,2,3)],[(1,4,3,2)]\}
  $$
  with representatives depicted below:
  \begin{center}
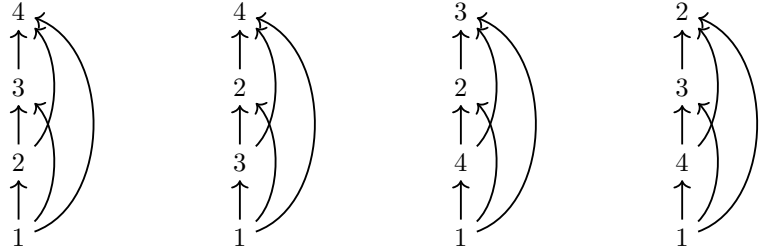
  \begin{tabular}{cccc}
  
\begin{tikzpicture}
\node(1) at (0,0) {$1$};
\node(2) at (0,1) {$2$};
\node(3) at (0,2) {$3$};
\node(4) at (0,3){$4$};

\draw[line width = 0.25 mm, ->] (1) -- (2);
\draw[line width = 0.25 mm, ->] (2) -- (3);
\draw[line width = 0.25 mm, ->] (3) -- (4);

\draw[line width = 0.25 mm, ->] (1) to [out=45,in=315, looseness = 0.8]  (3);
\draw[line width = 0.25 mm, ->] (2) to [out=45,in=315, looseness = 0.8] (4);
\draw[line width = 0.25 mm, ->] (1) to [out=20,in=340] (4);
\end{tikzpicture} \hspace{1 cm} &

\begin{tikzpicture}
\node(1) at (0,0) {$1$};
\node(3) at (0,1) {$3$};
\node(2) at (0,2) {$2$};
\node(4) at (0,3){$4$};

\draw[line width = 0.25 mm, ->] (1) -- (3);
\draw[line width = 0.25 mm, ->] (3) -- (2);
\draw[line width = 0.25 mm, ->] (2) -- (4);

\draw[line width = 0.25 mm, ->] (1) to [out=45,in=315, looseness = 0.8]  (2);
\draw[line width = 0.25 mm, ->] (3) to [out=45,in=315, looseness = 0.8] (4);
\draw[line width = 0.25 mm, ->] (1) to [line width = 0.25 mm, out=20,in=340] (4);
\end{tikzpicture} \hspace{1 cm} &

\begin{tikzpicture}
\node(1) at (0,0) {$1$};
\node(4) at (0,1) {$4$};
\node(2) at (0,2) {$2$};
\node(3) at (0,3) {$3$};

\draw[line width = 0.25 mm, ->] (1) -- (4);
\draw[line width = 0.25 mm, ->] (4) -- (2);
\draw[line width = 0.25 mm, ->] (2) -- (3);
\draw[line width = 0.25 mm, ->] (1) to  [out=45,in=315, looseness = 0.8] (2);
\draw[line width = 0.25 mm, ->] (4) to  [out=45,in=315, looseness = 0.8] (3);
\draw[line width = 0.25 mm, ->] (1) to [line width = 0.25 mm, out=20,in=340] (3);
\end{tikzpicture} \hspace{1 cm} &

\begin{tikzpicture}
\node(1) at (0,0) {$1$};
\node(4) at (0,1) {$4$};
\node(3) at (0,2) {$3$};
\node(2) at (0,3) {$2$};

\draw[line width = 0.25 mm, ->] (1) -- (4);
\draw[line width = 0.25 mm, ->] (4) -- (3);
\draw[line width = 0.25 mm, ->] (3) -- (2);
\draw[line width = 0.25 mm, ->] (1) to  [out=45,in=315, looseness = 0.8] (3);
\draw[line width = 0.25 mm, ->] (4) to  [out=45,in=315, looseness = 0.8] (2);
\draw[line width = 0.25 mm, ->] (1) to [line width = 0.25 mm, out=20,in=340] (2);
\end{tikzpicture}.
  \end{tabular}
  \end{center} 
  \end{example}
 
In this work, we define a toric analogue of Greene's rational function for toric posets. Just as Greene's rational function is a sum of rational functions indexed by the set of linear extensions of a poset, the toric analogue is a sum of rational functions indexed by the set of toric total extensions.

\begin{definition} \rm \label{def: toric Greene}
Let $[Q]$ be a toric poset. Then, we define $\psitor$ as 
\[\psitor:= \sum_{[w] \in \torext} \Psi_{\rm{tor}}^{[w]}(\mathbf{x}),\]
where
\[\Psi_{\rm{tor}}^{[w]}(\mathbf{x})=\frac{1}{(x_{w_1}-x_{w_2})(x_{w_2}-x_{w_3}) \cdots (x_{w_{n-1}}-x_{w_n})(x_{w_n}-x_{w_1})}.\]
\end{definition}

\begin{example} \rm Let $[Q]$ be the toric poset in Example \ref{Ex: toric poset example}. 
Since \[\torext=\{[(1,2,3,4)],[(1,3,2,4)],[(1,4,2,3)],[(1,4,3,2)]\},\] one has
\begin{align*}
\psitor &= \frac{1}{(x_1-x_2)(x_2-x_3)(x_3-x_4)(x_4-x_1)}  + \frac{1}{(x_1-x_3)(x_3-x_2)(x_2-x_4)(x_4-x_1)} \\ &+ 
\frac{1}{(x_1-x_4)(x_4-x_2)(x_2-x_3)(x_3-x_1)}+\frac{1}{(x_1-x_4)(x_4-x_3)(x_3-x_2)(x_2-x_1)} \\ &= \frac{-2}{(x_1-x_2)(x_1-x_3)(x_2-x_4)(x_3-x_4)}.
\end{align*}
\end{example}

\vskip.1in
\noindent
\textbf{Main Results.}  
In this paper, we use Greene's results as well as the results of Boussicault, F\'eray, Lascoux, and Reiner as motivation, and prove similar results for $\psitor$. The first is the following vanishing result.

\begin{theorem}
\label{thm: Cut vertex}
Let $[Q]$ be a toric poset, and let $G$ be the underlying undirected graph of $[Q]$. If $G$ is either disconnected with at least three vertices or has a cut vertex, then $\Psi_{\rm{tor}}^{[Q]}(\mathbf{x})=0$. 
\end{theorem}

In addition, Boussicault, F\'eray, Lascoux, and Reiner characterize the smallest denominator of $\Psi^{P}(\mathbf{x})$ as $\prod_{i < j}(x_i-x_j)$ where
the product runs over directed edges $i \rightarrow j$
in the Hasse diagram $H(P)$.
Section \ref{sec: properties of toric analogue} below will recall the \textit{toric Hasse diagram} $[Q]_{\rm{Hasse}}$ for a toric poset $[Q]$, leading to the following result on the denominator of $\psitor$.

\begin{theorem}
\label{thm: tor denom}
For $[Q]$ a toric poset, $\psitor$ can always be expressed over the denominator of 
\[\displaystyle \prod_{\{i,j\} \in [Q]_{\rm{Hasse}}}(x_i-x_j)\] where we take the product over all edges $\{i,j\}$ in $[Q]_{\rm{Hasse}}$.
\end{theorem} 
We note that the denominator in Theorem \ref{thm: tor denom} is not necessarily the smallest (see Remark \ref{Remark: tordenom not necessarily smallest}).
In addition, we show (see Proposition \ref{bounded poset})  that for a certain family of toric posets $[Q]$, $\psitor$ is a multiple of $\Psi^{P}(\mathbf{x})$. Using this result, we are able to recover the Kleiss-Kuijf shuffle relations \cite{Kleiss} by evaluating $\psitor$ for a specific toric poset $[Q]$ from this family (see Figure \ref{fig:KK toric poset family}).

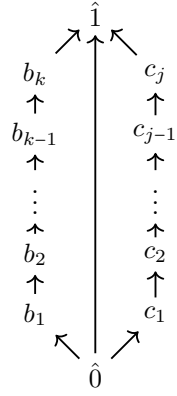
\begin{figure}[H]
\begin{center}
   \begin{tikzpicture}[scale = 0.8]
\node(2) at (0,0) {$\hat{0}$};
\node(3) at (-1,1) {$b_1$};
\node(4) at (-1,2) {$b_2$};
\node(5) at (-1,3) {$\vdots$};
\node(6) at (-1,4) {$b_{k-1}$};
\node(1) at (0,6) {$\hat{1}$};
\node(7) at (-1,5) {$b_k$};
\draw[line width = 0.25 mm, ->] (2) -- (3);
\draw[line width = 0.25 mm, ->] (3) -- (4);
\draw[line width = 0.25 mm, ->] (4) -- (5);
\draw[line width = 0.25 mm, ->] (-1,3.4) -- (6);
\draw[line width = 0.25 mm, ->] (6) -- (7);
\draw[line width = 0.25 mm, ->] (7) -- (1);
\draw[line width = 0.25 mm, ->] (2) -- (1);
\node(8) at (1,1) {$c_1$};
\node(9) at (1,2) {$c_2$};
\node(10) at (1,3) {$\vdots$};
\node(11) at (1,4) {$c_{j-1}$};
\node(12) at (1,5) {$c_j$};
\draw[line width = 0.25 mm, ->] (2) -- (8);
\draw[line width = 0.25 mm, ->] (8) -- (9);
\draw[line width = 0.25 mm, ->] (9) -- (10);
\draw[line width = 0.25 mm, ->] (1,3.4) -- (11);
\draw[line width = 0.25 mm, ->] (11) -- (12);
\draw[line width = 0.25 mm, ->] (12) -- (1);
\end{tikzpicture} 
    \caption{One representative of $[Q]$ from Corollary \ref{cor: Kleiss-Kuijf relations}.}
    \label{fig:KK toric poset family}
    \end{center}
\end{figure}

Let $\mathbf{b}=(b_1,b_2, \ldots, b_k)$, $\mathbf{c}= (c_1,c_2,\ldots, c_j)$, and $\mathrm{rev}(\mathbf{b})=(b_k,\ldots,b_2,b_1)$. As convention, let $b_{k+1}=c_{j+1}= \hat{1}$ and $b_0=c_0=\hat{0}$.
 The \textit{shuffle set} $\mathbf{b} \shuffle \mathbf{c}$ is the set of all permutations of $(b_1,\ldots,b_k,c_1,\ldots,c_j)$ such that the subsequences of the $b_i$ and $c_i$ appear in the same order as in $\mathbf{b}$ and $\mathbf{c}$ respectively.

\begin{cor}({Kleiss-Kuijf Shuffle Relations}) \label{cor: Kleiss-Kuijf relations}
 For $\psitor$ where $[Q]$ is the toric poset in Figure \ref{fig:KK toric poset family},
$$
\psitor = \frac{(-1)^k}{\displaystyle {\prod_{r=0}^k (x_{b_{r+1}}-x_{b_r}) \cdot \prod_{s=0}^j (x_{c_s}-x_{c_{s+1}}) }}$$
or equivalently,

\begin{align*}
\sum_{\mathbf{a} \in \mathbf{b} \shuffle  \mathbf{c}} \Psi_{\rm{tor}}^{[(\hat{1},\hat{0},\mathbf{a})]}(\mathbf{x})  &= (-1)^k \Psi_{\rm{tor}}^{[(\hat{1},\rm{rev}(\mathbf{b}),\hat{0},\mathbf{c})]}(\mathbf{x}).
\end{align*}
\end{cor}

In order to compute $\psitor$ for a toric poset $[Q]$, it is necessary to compute the set $\torext$ of toric total extensions. 
We will show in 
Theorem~\ref{thm: sharp-P-complete} that counting $\torext$ is a $\# P$-complete problem, so that one should not expect efficient algorithms for finding this set.  For theoretical purposes, we will often use a decomposition (see Proposition \ref{prop: Toric lin ext reformulations} part (ii)) that expresses $\torext$ as a disjoint union indexed by the subset $[Q]_v$ of quivers in $[Q]$ having a particular chosen vertex $v$ as a source. 
 Unfortunately, we have no efficient algorithm for computing this subset $[Q]_v$.
Nevertheless, we will derive a somewhat more efficient recursive algorithm (see Theorem \ref{thm:unified-recursion-steps}) to
compute $\torext$, motivated by the following simple and well-known recursive description of the set $\mathcal{L}(P)$ of linear extensions of a poset $P$.

\begin{lemma} \label{lem: basic poset fact}
Let $P$ be a poset, and let $a,b$ be two incomparable elements of $P$. Then, 
\[
\mathcal{L}(P) = \mathcal{L}(P_{a \rightarrow b}) \sqcup \mathcal{L}(P_{b \rightarrow a})
\]
where $P_{a \rightarrow b}$ is obtained from $P$ by adding the relation $a < b$ and $P_{b \rightarrow a}$ is defined similarly.
\end{lemma}
Our recursion for
$\torext$
will rely on the
following result, which may be of independent interest.

\begin{theorem}
\label{thm: Source-sink equiv with fixed source} 
Let $v$ be any vertex in an acyclic quiver $Q$, and let $Q_1,Q_2$ be any two acyclic quivers in the subset $[Q]_v$ of the source-sink flip-equivalence class $[Q]$, so $v$ is a source in both $Q_1$ and $Q_2$.  

Then there exists a source-sink flip sequence from $Q_1$ to $Q_2$ such that every intermediate quiver in the sequence also has $v$ as a source. In other words, the flip sequence  does not flip at $v$, nor at neighbors of $v$.
\end{theorem}

\begin{remark} \rm The rational function $\Psi_{\rm{tor}}^{[w]}(\mathbf{x})$ appears in scattering amplitude computations as a {\it Parke-Taylor factor} \cite{parketaylor}, as we explain here. Recently, in \cite{Williams}, Parisi, Sherman-Bennett, Tessler, and Williams utilized $\Psi_{\rm{tor}}^{[w]}(\mathbf{x})$ in order to prove a tiling conjecture for the $m=2$ {\it amplituhedron}.
Consider a generic $2 \times n$ matrix

\[A = \begin{pmatrix}
a_{11} & a_{12} & \cdots & a_{1n} \\
a_{21} & a_{22} & \cdots & a_{2n} \\
\end{pmatrix} \in Gr_{2,n}.\]
Then, a Pl\"ucker coordinate $P_{ij}$ in the Grassmannian $Gr_{2,n}$ is the determinant  $\mathrm{det}\begin{pmatrix}
a_{1i} & a_{1j} \\
a_{2i} & a_{2j} \\
\end{pmatrix}=a_{1i}a_{2j}-a_{1j}a_{2i}$. In \cite{Williams}, the authors begin with the following \textit{Parke-Taylor function}
\[\mathrm{PT}(w):= \frac{1}{P_{w_1w_2}P_{w_2w_3}\cdots P_{w_nw_1}}\]
where each $P_{ij}$ is a Pl\"ucker coordinate in $Gr_{2,n}$. 
In their Remark 5.2 and proof of Proposition 5.4, they show how a point in $\widehat{Gr}^{\circ}_{2,n}$, the dense subset of $Gr_{2,n}$ where all Pl\"ucker coordinates are non-vanishing, can be represented with the $2 \times n$ matrix
\[\begin{pmatrix}
1 & 1 & \cdots & 1 \\
x_1 & x_2 & \cdots & x_n
\end{pmatrix}. \]
Then, the Pl\"ucker coordinate $P_{ij}$ of this matrix is the linear factor $x_j-x_i$. Utilizing this fact, they are able to rewrite $\mathrm{PT}(w)$ as the rational function
\[
\mathrm{PT}(w) = \frac{1}{(x_{w_2}-x_{w_1})(x_{w_3}-x_{w_2})\cdots(x_{w_n}-x_{w_{n-1}})(x_{w_1}-x_{w_n})}, \label{PT(w)}
\]
which we recognize (up to $\pm$ sign) as $\Psi_{\rm{tor}}^{[w]}(\mathbf{x})$ in Definition \ref{def: toric Greene}.

Moreover, they use \textit{cyclic extensions} of partial cyclic orders in order to give volume formulas for \textit{Parke-Taylor polytopes} and certain positroid polytopes. Since a toric total extension can be seen to be the same as a total cyclic order \cite{toric}, we are hopeful that our identities proven for $\psitor$ could be useful in similar settings. We further discuss the connection and distinction between $\psitor$ and their work in Section \ref{sec: relationship to tricolored subdivisions}.
\end{remark}

\noindent \textbf{Outline of paper.} Section \ref{sec:posets and hyperplanes} discusses the association between posets and chambers in graphic hyperplane arrangements. Section \ref{sec: toric posets} provides the association between toric posets and chambers in toric graphic hyperplane arrangements. We also discuss properties of toric posets analogous to those of ordinary posets. Section \ref{sec: properties of toric analogue} proves Theorem \ref{thm: Cut vertex}, Theorem \ref{thm: tor denom}, and Corollary \ref{cor: Kleiss-Kuijf relations}. In Section \ref{sec: relationship to tricolored subdivisions}, we discuss how the identities proven in Section \ref{sec: properties of toric analogue} relate to work of Parisi, Sherman-Bennett, Tessler, and Williams  in \cite{Williams}. In Section \ref{Sec: source-sink equivalence}, we prove Theorem \ref{thm: Source-sink equiv with fixed source} and in Section \ref{sec: algorithm}, we provide a recursive algorithm for finding the set of toric total extensions of a toric poset.

\subsection*{Acknowledgements}

The author is very grateful to Vic Reiner for guidance throughout all stages of this project. The author would also like to thank Esther Banaian, Swee Hong Chan, Patricia Commins, Colin Defant, Nick Early, Tucker Ervin, Darij Grinberg, Matt Macauley, and Scott Neville for helpful conversations and references, and Son Nguyen, for his assistance in creating a Sage program to acquire initial data. The author would also like to thank Gregg Musiker for his encouragement of this work.  Work partially supported by NSF DMS-2053288.

\section{Posets and Graphic Hyperplane Arrangements} \label{sec:posets and hyperplanes}
The definition of a toric poset relies on the well-studied association between posets and chambers in graphic hyperplane arrangements  \cite{toric, Zaslavsky, postnikovfaces, stanleyacyclic}, so we first discuss this correspondence. A poset $P$ on $[n]$ gives rise to an open polyhedral cone $c(P)$ in $\mathbb{R}^n$, where
\[
c(P) := \{x \in \mathbb{R}^n: x_i < x_j \text{ if } i <_P j\}.
\]

Connected components in the complement of a graphic hyperplane arrangement are open polyhedral cones called \textit{chambers}, and each cone $c(P)$ appears as a chamber in the complement of at least one graphic hyperplane arrangement. We review how to construct a graphic hyperplane arrangement from a simple graph $G$, i.e. one with no loops nor parallel edges. 

Let $G$ be a simple, undirected graph on the vertex set $[n]$, so $G \subseteq \binom{[n]}{2}$. Then, the \textit{graphic hyperplane arrangement} $\mathcal{A}(G)$ is defined to be
\[ \mathcal{A}(G):=\bigcup_{\{i,j\} \in G} \mathcal{H}_{ij}\]
where $\mathcal{H}_{ij}$ is the hyperplane $x_i=x_j$.

A quiver is a directed graph, and an acyclic quiver is one that contains no directed cycles. There is a one-to-one
 correspondence between chambers in $\mathbb{R}^n - \mathcal{A}(G)$ and acyclic quivers that have the same underlying graph $G$.  Given such a chamber, for every pair of vertices $i,j$ such that $\{i,j\} \in G$, we orient this edge $i \to j$ if $x_i < x_j$ and orient the edge $j \to i$ otherwise. It follows that this quiver is acyclic. Moreover, any acyclic quiver on $n$ vertices induces a poset structure on $n$ elements. In particular, we set $i < j$ in the poset whenever there is a directed path from $i$ to $j$ in the quiver.

\begin{example} \rm We show two graphs on three vertices and consider their associated graphic hyperplane arrangements. For each arrangement, we label the chambers by the posets induced by acyclic orientations of the corresponding graph and each picture is drawn inside the two-dimensional hyperplane $x_1+x_2+x_3=0$ in $\mathbb{R}^3$. 
\begin{center}
\begin{tabular}{cc}

\begin{tikzpicture}[scale = 2]
\draw[blue, line width = 0.35 mm, <->] (-0.5,-0.866) -- (0.5,  0.866) node[anchor=west, yshift = 0.4 cm, xshift = -0.6 cm] {$x_2=x_3$};
\draw[line width = 0.35 mm, blue, <->] (-0.5,0.866) -- (0.5,  -0.866) node[anchor=west, yshift= -0.55 cm, xshift = -0.6 cm ] {$x_1=x_3$};
\draw[blue, line width = 0.35 mm, <->] (-1,0) -- (1,0) node[anchor=west] {$x_1=x_2$};
\node(0) at (-0.65,0.5){\begin{tikzpicture}[scale = 1]
\node(3) at (0,0){$3$};
\node(1) at (0,0.75){$1$};
\node(2) at (0,1.5){$2$};
\draw[->] (3)  -- (1);
\draw[->] (1)  -- (2);
\end{tikzpicture}};

\node(1) at (0,0.6) {\begin{tikzpicture}[scale = 1]
\node(1) at (0,0){$1$};
\node(3) at (0,0.75){$3$};
\node(2) at (0,1.5){$2$};
\draw[->] (1)  -- (3);
\draw[->] (3)  -- (2);
\end{tikzpicture}};

\node(3) at (0.65,0.5) {\begin{tikzpicture}[scale = 1]
\node(1) at (0,0){$1$};
\node(2) at (0,0.75){$2$};
\node(3) at (0,1.5){$3$};
\draw[->] (1)  -- (2);
\draw[->] (2)  -- (3);
\end{tikzpicture}};

\node(4) at (0,-0.65) {\begin{tikzpicture}[scale = 1]
\node(2) at (0,0){$2$};
\node(3) at (0,0.75){$3$};
\node(1) at (0,1.5){$1$};
\draw[->] (2)  -- (3);
\draw[->] (3)  -- (1);
\end{tikzpicture}};

\node(5) at (-0.65,-0.55) {\begin{tikzpicture}[scale = 1]
\node(3) at (0,0){$3$};
\node(2) at (0,0.75){$2$};
\node(1) at (0,1.5){$1$};
\draw[->] (3)  -- (2);
\draw[->] (2)  -- (1);
\end{tikzpicture}};

\node(6) at (0.65,-0.55) {\begin{tikzpicture}[scale = 1]
\node(2) at (0,0){$2$};
\node(1) at (0,0.75){$1$};
\node(3) at (0,1.5){$3$};
\draw[->] (2)  -- (1);
\draw[->] (1)  -- (3);
\end{tikzpicture}};

\node(7) at (-1.5,0) {\begin{tikzpicture}
\node(1) at (-0.5,0){$1$};
\node(2) at (0.5,0){$2$};
\node(3) at (0,1){$3$};
\draw[line width = 0.35 mm, -] (1)  -- (3);
\draw[line width = 0.35 mm, -] (2)  -- (3);
\draw[line width = 0.35 mm, -] (1)  -- (2);
\end{tikzpicture}};

\end{tikzpicture} &

\begin{tikzpicture}[scale = 2]
\draw[line width = 0.35 mm, <->, dotted, blue, opacity = 0.4] (-0.5,-0.866) -- (0.5,  0.866) node[opacity = 1.0, anchor=west, yshift = 0.4 cm, xshift = -0.6 cm] {$x_2=x_3$};
\draw[blue, line width = 0.35 mm, <->] (-0.5,0.866) -- (0.5,  -0.866) node[anchor=west, yshift= -0.55 cm, xshift = -0.6 cm]
{$x_1=x_3$};
\draw[blue, line width = 0.35 mm, <->] (-1,0) -- (1,0) node[anchor=west] {$x_1=x_2$};
\node(0) at (-1.5,0){\begin{tikzpicture}
\node(1) at (-0.5,0){$1$};
\node(2) at (0.5,0){$2$};
\node(3) at (0,1){$3$};
\draw[line width = 0.35 mm, -] (1)  -- (3);
\draw[line width = 0.35 mm, -] (1)  -- (2);
\end{tikzpicture}};

\node(1) at (-0.65,0.5){\begin{tikzpicture}[scale = 1]
\node(3) at (0,0){$3$};
\node(1) at (0,0.75){$1$};
\node(2) at (0,1.5){$2$};
\draw[->] (3)  -- (1);
\draw[->] (1)  -- (2);
\end{tikzpicture}};

\node(2) at (-0.4, -0.5){\begin{tikzpicture}[scale = 1]
\node(2) at (-0.5,0){$2$};
\node(3) at (0.5,0){$3$};
\node(1) at (0,1){$1$};
\draw[->] (2)  -- (1);
\draw[->] (3)  -- (1);
\end{tikzpicture}};

\node(3) at (0.4, 0.5){\begin{tikzpicture}[scale = 1]
\node(2) at (-0.5,1){$2$};
\node(3) at (0.5,1){$3$};
\node(1) at (0,0){$1$};
\draw[ ->] (1)  -- (2);
\draw[ ->] (1)  -- (3);
\end{tikzpicture}};

\node(6) at (0.65,-0.55) {\begin{tikzpicture}[scale = 1]
\node(2) at (0,0){$2$};
\node(1) at (0,0.75){$1$};
\node(3) at (0,1.5){$3$};
\draw[->] (2)  -- (1);
\draw[->] (1)  -- (3);
\end{tikzpicture}};
\end{tikzpicture}
\end{tabular}
\end{center}
\end{example}

\noindent In addition, a poset $P$ is also determined by its set of linear extensions. Each linear extension $(w_1, w_2, \ldots, w_n)$ corresponds to a chamber 
\[
c_w:= \{\mathbf{x} \in \mathbb{R}^n: x_{w_1} < x_{w_2} < \cdots < x_{w_n}\}
\]
in the complement of the complete graphic hyperplane arrangement $\mathcal{A}(K_n)$, also known as the braid arrangement. From this observation, we have
\begin{equation}
\overline{c(P)}=\bigcup_{w \in \mathcal{L}(P)} \overline{c}_w, \label{cones of braid arrangement}
\end{equation}
where $\overline{(\cdot)}$ denotes topological closure.

\begin{example} \rm
Let $P$ be the partial order on the set $\{1,2,3\}$, where $1<_P 2$ and $1<_P 3$. Consider the Hasse diagram $H(P)$, which is an acyclic quiver. We draw the graphic hyperplane arrangement for $H(P)$ as well the braid arrangement/graphic hyperplane arrangement for the complete graph on 3 vertices. Note that we draw each picture inside the two-dimensional hyperplane $x_1+x_2+x_3=0$ in $\mathbb{R}^3$.
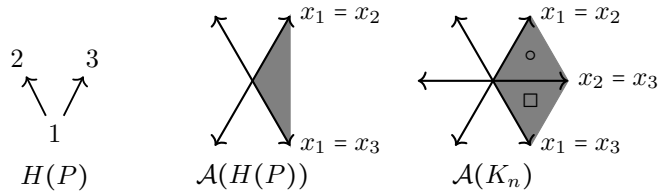
\begin{center} 
\begin{tabular}{ccccc}
\begin{tikzpicture}
\node(1) at (0,0){1};
\node(2) at (-0.5,1){2};
\node(3) at (0.5,1){3};
\draw[line width = 0.25 mm, ->] (1)  -- (2);
\draw[line width = 0.25 mm, ->](1) -- (3);
\node[label = $H(P)$] at (0,-1) {};
\end{tikzpicture}
& & &

\begin{tikzpicture}[scale = 1]
\filldraw[gray,opacity=0.3] (-0.5,0.866) -- (0,0) 
-- (1,0);
\draw[black, thick, <->] (-1,0) -- (1,  0) node[anchor=west] {$x_1=x_2$};
\draw[black, thick, <->] (-0.5,0.866) -- (0.5,  -0.866) node[anchor=west] {$x_1=x_3$};
\node[label = $\mathcal{A}(H(P))$] at (0,-1.7) {};
\end{tikzpicture} & 
\begin{tikzpicture}[scale = 1]
\filldraw[gray,opacity=0.3] (0,0) -- (1,0) -- (0.5,0.866);
\filldraw[gray,opacity=0.3] (-0.5,0.866) -- (0,0) -- (0.5,0.866);
\draw[black, thick, <->] (-0.5,-0.866) -- (0.5,  0.866) node[anchor=west] {$x_2=x_3$};
\draw[black, thick, <->] (-0.5,0.866) -- (0.5,  -0.866) node[anchor=west] {$x_1=x_3$};
\draw[black, thick, <->] (-1,0) -- (1,  0) node[anchor=west] {$x_1=x_2$};
\node[label = $\mathcal{A}(K_3)$] at (0,-1.7) {};
\node[label = $\circ$] at (0.5,0) {};
\node[label = $\square$] at (0,0.2) {};
\end{tikzpicture}
\end{tabular}
\end{center}
We let the shaded region in $\mathbb{R}^3-\mathcal{A}(H(P))$ mark the region where $x_2 > x_1$ and $x_3 > x_1$. In addition, we let the shaded region in $\mathbb{R}^3-\mathcal{A}(K_3)$ that is marked with a circle denote the region where $x_3>x_2>x_1$ and the shaded region marked with a square denote where $x_2 > x_3 > x_1$. Since $\mathcal{L}(P) = \{(1,2,3),(1,3,2)\}$, we see that Equation \eqref{cones of braid arrangement} holds.
\end{example} 
Equation $\eqref{cones of braid arrangement}$ demonstrates that when one fixes the graph $G$, posets (chambers) are determined by their sets of linear extensions. Posets may arise as chambers in several graphic hyperplane arrangements as the graph $G$ varies.

\begin{example} \rm
We consider two different quivers on three elements that induce the same poset $P$ and we compare their graphic hyperplane arrangements. Let $G$ be the underlying graph of $Q$ and let $G'$ be the underlying graph of $Q'$. Although the graphs are different, thereby forcing the graphic hyperplane arrangements to be different, the closures of the cones in the complement of each arrangement are the same.

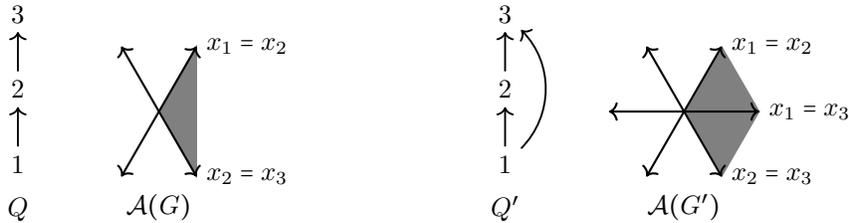
\begin{center}
\begin{tabular}{ccccccc}
\begin{tikzpicture}
\node(1) at (0,0){1};
\node(2) at (0,1){2};
\node(3) at (0,2){3};
\draw[line width = 0.25 mm, ->] (1)  -- (2);
\draw[line width = 0.25 mm, ->](2) -- (3);
\node[label = $Q$] at (0,-1) {};
\end{tikzpicture}
&&&

\begin{tikzpicture}[scale = 1]
\filldraw[gray,opacity=0.3] (0,0) -- (1,0) -- (0.5,0.866);
\draw[black, thick, <->] (-0.5,-0.866) -- (0.5,  0.866) node[anchor=west] {$x_2=x_3$};
\draw[black, thick, <->] (-1,0) -- (1,  0) node[anchor=west] {$x_1=x_2$};
\node[label = $\mathcal{A}(G)$] at (0,-1.7) {};
\end{tikzpicture}

& 
\hspace{2 cm}

\begin{tikzpicture}
\node(1) at (0,0){1};
\node(2) at (0,1){2};
\node(3) at (0,2){3};
\draw[line width = 0.25 mm, ->] (1)  -- (2);
\draw[line width = 0.25 mm, ->](2) -- (3);
\draw[line width = 0.25 mm, ->] (1) to [out = 45, in = 315] (3);
\node[label = $Q'$] at (0,-1) {};
\end{tikzpicture}

&&

\begin{tikzpicture}[scale = 1]
\filldraw[gray,opacity=0.3] (0,0) -- (1,0) -- (0.5,0.866);
\draw[black, thick, <->] (-0.5,-0.866) -- (0.5,  0.866) node[anchor=west] {$x_2=x_3$};
\draw[black, thick, <->] (-0.5,0.866) -- (0.5,  -0.866) node[anchor=west] {$x_1=x_3$};
\draw[black, thick, <->] (-1,0) -- (1,  0) node[anchor=west] {$x_1=x_2$};
\node[label = $\mathcal{A}(G')$] at (0,-1.7) {};
\end{tikzpicture}

\end{tabular}
\end{center}

\end{example}
Thus, there is ambiguity when identifying $P$ with an {\it acyclic quiver} $Q$. Although there is ambiguity, there are two natural choices when identifying a poset $P$ with
such an acyclic quiver:
\begin{enumerate}
\item the {\it transitive closure} $\overline{P}$, whose directed edges  $i \rightarrow j$ are the relations $i < j$ in $P$.

\item the {\it Hasse diagram} $H(P)$, whose directed edges  $i \rightarrow j$ are the covering relations $i \lessdot j$ in $P$.
\end{enumerate}

\begin{example} \rm
We look at three quivers that induce the same poset $P$.

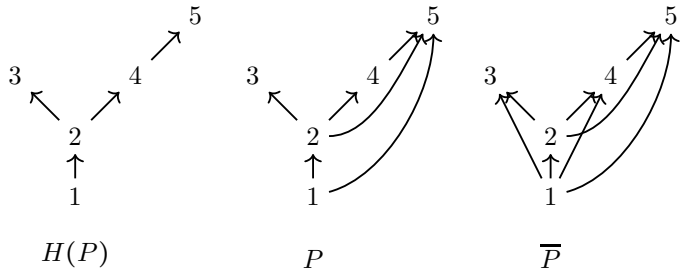
\begin{center}
\begin{tabular}{ccc}
\begin{tikzpicture}[scale = 0.8]
\node(1) at (0,0) {$1$};
\node(2) at (0,1) {$2$};
\node(3) at (-1,2) {$3$};
\node(4) at (1,2) {$4$};
\node(5) at (2,3) {$5$};
\node[label = $H(P)$] at (0,-1.5) {};
\draw[line width = 0.25 mm,->] (1) -- (2);
\draw[line width = 0.25 mm,->] (2) -- (3);
\draw[line width = 0.25 mm,->] (2) -- (4);
\draw[line width = 0.25 mm,->] (4) -- (5);
\end{tikzpicture} & 
\begin{tikzpicture}[scale = 0.8]
\node(1) at (0,0) {$1$};
\node(2) at (0,1) {$2$};
\node(3) at (-1,2) {$3$};
\node(4) at (1,2) {$4$};
\node(5) at (2,3) {$5$};
\node[label = $P$] at (0,-1.5) {};
\draw[line width = 0.25 mm,->] (1) -- (2);
\draw[line width = 0.25 mm,->] (2) -- (3);
\draw[line width = 0.25 mm,->] (2) -- (4);
\draw[line width = 0.25 mm,->] (4) -- (5);
\draw[line width = 0.25 mm,->] (1) to [out = 15, in = 270, looseness = 0.8] (5);
\draw[line width = 0.25 mm,->] (2) to [out = 0, in = 240, looseness = 0.8] (5);
\end{tikzpicture}  & 
\begin{tikzpicture}[scale=0.8]
\node(1) at (0,0) {$1$};
\node(2) at (0,1) {$2$};
\node(3) at (-1,2) {$3$};
\node(4) at (1,2) {$4$};
\node(5) at (2,3) {$5$};
\node[label = $\overline{P}$] at (0,-1.5) {};
\draw[line width = 0.25 mm,->] (1) -- (2);
\draw[line width = 0.25 mm,->] (1) -- (4);
\draw[line width = 0.25 mm,->] (2) -- (3);
\draw[line width = 0.25 mm,->] (2) -- (4);
\draw[line width = 0.25 mm,->] (4) -- (5);
\draw[line width = 0.25 mm,->] (1) -- (3);
\draw[line width = 0.25 mm,->] (1) to [out = 15, in = 270, looseness = 0.8] (5);
\draw[line width = 0.25 mm,->] (2) to [out = 0, in = 240, looseness = 0.8] (5);
\end{tikzpicture}
\end{tabular}
\end{center}
\end{example}

\section{Toric Posets} \label{sec: toric posets}
In \cite{toric}, Develin, Macauley, and Reiner introduce toric posets. Throughout the paper, we may distinguish toric posets from posets, by calling posets ``ordinary" posets. In order to define toric posets, we begin with toric graphic hyperplane arrangements which are the source of the name ``toric" poset. Given an undirected graph $G$ on $n$ vertices, we saw before that there is an associated graphic hyperplane arrangement $\mathcal{A}(G)$ inside $\mathbb{R}^n$. We define a quotient map 
$\pi: \mathbb{R}^n \rightarrow \mathbb{R}^n/\mathbb{Z}^n$. The \textit{toric graphic hyperplane arrangement} associated to $G$ is 
\[\mathcal{A}_{\rm{tor}}(G)=\pi(\mathcal{A}(G)).\] A connected component of $\mathbb{R}^n/\mathbb{Z}^n - \mathcal{A}_{\rm{tor}}(G)$ is a \textit{toric chamber}. A \textit{toric poset} is a set that arises as a toric chamber in a toric graphic hyperplane arrangement for at least one graph $G$.

Naturally, given $\mathbf{x}, \mathbf{y} \in \mathbb{R}^n$, we know that these points lie in the same equivalence class in $\mathbb{R}^n/\mathbb{Z}^n$ exactly when for each coordinate $1 \leq i \leq n$, we have $x_i \mod 1 = y_i \mod 1$. Therefore, we can still recover an acyclic quiver with underlying graph $G$ for each point $[\mathbf{x}] \in \mathbb{R}^n/\mathbb{Z}^n$ by orienting $\{i,j\} \in G$ as $i \to j$ if $x_i \mod 1 < x_j \mod 1$ and orienting $\{i,j\}$ as $j \rightarrow i$  otherwise.

\begin{quote}
\textit{Key point:} By this construction, two points in the same toric chamber do not necessarily map to the same acyclic quiver. To account for this, the following \textit{flip operation} is defined.
\end{quote}

\begin{definition}[\cite{toric}]
\rm
Consider acyclic quivers $Q_1, Q_2$ that differ by converting one source vertex (all edges directed outward) to one sink vertex (all edges directed inward). Then, we say that $Q_1,Q_2$ differ by a \textit{flip}. This flip operation induces an equivalence relation on the set of acyclic quivers with the same underlying graph $G$, and we denote this equivalence relation as $\equiv$.

\end{definition}
\begin{remark}  \label{Rmk: literature review} \rm
This flip operation was studied by Mosesian and Pretzel in \cite{mosesian} and \cite{pretzel}, respectively. Moreover, this flip operation has appeared in other works including Chen \cite{chen}, Defant and Kravitz \cite{defant}, Eriksson and Eriksson \cite{eriksson}, Macauley and Mortveit \cite{macauley2009posets}, Speyer \cite{speyer}, and Propp \cite{Propp}. This flip operation also appears in the context of reflection functors in quiver representations \cite{bernstein}.

These flips are an instance of quiver mutation at a sink or source vertex \cite{fomin2}. Caldero and Keller show that for two mutation-equivalent acyclic quivers $Q_1,Q_2$, there exists a source-sink flip sequence one can apply to $Q_1$ that yields a quiver isomorphic after relabeling vertices to $Q_2$, but not necessarily equal to $Q_2$ \cite[Cor. 4]{caldero} (also seen in \cite{fomin2017introduction}). For the definition of mutation, see \cite{williams2014cluster}. In Example \ref{Ex. caldero-keller}, we show an example of this result. These connections are our motivation to refer to directed graphs as ``quivers."
\end{remark}

With ordinary posets, we saw that there is a bijection between chambers of $\mathcal{A}(G)$ and the set of acyclic quivers with underlying graph $G$. For toric posets, we have the following theorem.
\begin{theorem} \emph{(\cite[Thm. 1.4]{toric})}
    There is a bijection between the chambers of $\mathcal{A}_{\rm{tor}}(G)$ and the set of acyclic quivers with underlying graph $G$ equipped with $\equiv$.
\end{theorem}

With $\equiv$ defined, we can define toric posets in a combinatorial way. 
\begin{definition} \rm
A \textit{toric poset} $[Q]$ is an equivalence class of acyclic quivers that are equivalent under the relation of flipping a sink vertex to a source vertex and vice versa.
\end{definition}
We recall Example \ref{Ex: toric poset example} for an example of a toric poset.

\begin{example} \rm \label{Ex. caldero-keller} In the graph below, we show all quivers that are mutation equivalent to the quiver $1 \rightarrow 2 \rightarrow 3$. We draw an edge labeled $\mu_i$ between two quivers if they are related by mutating at vertex $i$. We note that mutation is an involution. The subgraph $H$ highlighted in blue is the source-sink flip equivalence class of $1 \rightarrow 2 \rightarrow 3$, i.e. the toric poset that contains the quiver $1 \rightarrow 2 \rightarrow 3$. We emphasize that any two acyclic quivers in this picture can be related by a sequence of steps that involve both source-sink flips and relabeling vertices.

\begin{center}
\begin{tikzpicture}[scale = 1.2]
\draw (-6,0.95) circle (0.58cm);
\node(1) at (-6,1) {$\begin{tikzpicture}[scale = 0.85]
\node (60) at (0,1) {\small{$1$}};
\node (61) at (0.5,0) {\small{$2$}};
\node (62) at (-0.5,0) {\small{$3$}};
\draw[->] (60) -- (61);
\draw[->] (60) -- (62);
\end{tikzpicture} $};
\node(2) at (-4,1) {$\begin{tikzpicture}[scale = 0.85]
\node (60) at (0,1) {\small{$1$}};
\node (61) at (0.5,0) {\small{$2$}};
\node (62) at (-0.5,0) {\small{$3$}};
\draw[->] (61) -- (60);
\draw[->] (62) -- (60);
\end{tikzpicture} $};

\draw (-4,0.95) circle (0.58cm);

\node(3) at (-1,1) {$\begin{tikzpicture}[scale = 0.85]
\node (60) at (0,1) {\small{\color{blue}{$1$}}};
\node (61) at (0.5,0) {\small{\color{blue}{$2$}}};
\node (62) at (-0.5,0) {\small\color{blue}{{$3$}}};
\draw[->,blue] (60) -- (61);
\draw[->,blue] (62) -- (61);
\end{tikzpicture} $};

\draw[blue] (-1,0.95) circle (0.58cm);

\node(4) at (1,1) {$\begin{tikzpicture}[scale = 0.85]
\node (60) at (0,1) {\small{\color{blue}{$1$}}};
\node (61) at (0.5,0) {\small{\color{blue}{$2$}}};
\node (62) at (-0.5,0) {\small{\color{blue}{$3$}}};
\draw[->, blue] (61) -- (60);
\draw[->, blue] (61) -- (62);
\end{tikzpicture} $};

\draw[blue] (1,0.95) circle (0.58cm);

\node(5) at (4,1) {$\begin{tikzpicture}[scale = 0.85]
\node (60) at (0,1) {\small{$1$}};
\node (61) at (0.5,0) {\small{$2$}};
\node (62) at (-0.5,0) {\small{$3$}};
\draw[->] (62) -- (61);
\draw[->] (62) -- (60);
\end{tikzpicture} $};

\draw (4,0.95) circle (0.58cm);

\node(6) at (6,1) {$\begin{tikzpicture}[scale = 0.85]
\node (60) at (0,1) {\small{$1$}};
\node (61) at (0.5,0) {\small{$2$}};
\node (62) at (-0.5,0) {\small{$3$}};
\draw[->] (60) -- (62);
\draw[->] (61) -- (62);
\end{tikzpicture} $};
\draw (6,0.95) circle (0.58cm);
\node(7) at (-5,-1) {$\begin{tikzpicture}[scale = 0.85]
\node (60) at (0,1) {\small{$1$}};
\node (61) at (0.5,0) {\small{$2$}};
\node (62) at (-0.5,0) {\small{$3$}};
\draw[->] (60) -- (61);
\draw[->] (62) -- (60);
\end{tikzpicture} $};
\draw (-5,-1.1) circle (0.58cm);

\node(8) at (0,-1) {$\begin{tikzpicture}[scale = 0.85]
\node (60) at (0,1) {\small{\color{blue}{$1$}}};
\node (61) at (0.5,0) {\small{\color{blue}{$2$}}};
\node (62) at (-0.5,0) {\small{\color{blue}{$3$}}};
\draw[->, blue] (60) -- (61);
\draw[->, blue] (61) -- (62);
\end{tikzpicture} $};

\draw[blue] (0,-1.05) circle (0.58cm);
\node(9) at (5,-1) {$\begin{tikzpicture}[scale = 0.85]
\node (60) at (0,1) {\small{$1$}};
\node (61) at (0.5,0) {\small{$2$}};
\node (62) at (-0.5,0) {\small{$3$}};
\draw[->] (61) -- (62);
\draw[->] (62) -- (60);
\end{tikzpicture} $};
\draw (5,-1.05) circle (0.58cm);
\node(10) at (0,-3) {$\begin{tikzpicture}[scale = 0.85]
\node (60) at (0,1) {\small{$1$}};
\node (61) at (0.5,0) {\small{$2$}};
\node (62) at (-0.5,0) {\small{$3$}};
\draw[->] (60) -- (62);
\draw[->] (62) -- (61);
\draw[->] (61) -- (60);
\end{tikzpicture} $};
\draw (0,-3.05) circle (0.58cm);
\node(11) at (-5,3) {$\begin{tikzpicture}[scale = 0.85]
\node (60) at (0,1) {\small{$1$}};
\node (61) at (0.5,0) {\small{$2$}};
\node (62) at (-0.5,0) {\small{$3$}};
\draw[->] (61) -- (60);
\draw[->] (60) -- (62);
\end{tikzpicture} $};
\draw (-5,2.9) circle (0.58cm);
\node(12) at (0,3) {$\begin{tikzpicture}[scale = 0.85]
\node (60) at (0,1) {\small{\color{blue}{$1$}}};
\node (61) at (0.5,0) {\small{\color{blue}{$2$}}};
\node (62) at (-0.5,0) {\small{\color{blue}{$3$}}};
\draw[->, blue] (62) -- (61);
\draw[->, blue] (61) -- (60);
\end{tikzpicture} $};
\draw[blue] (0,2.9) circle (0.58cm);
\node(13) at (5,3) {$\begin{tikzpicture}[scale = 0.85]
\node (60) at (0,1) {\small{$1$}};
\node (61) at (0.5,0) {\small{$2$}};
\node (62) at (-0.5,0) {\small{$3$}};
\draw[->] (60) -- (62);
\draw[->] (62) -- (61);
\end{tikzpicture} $};
\draw (5,2.9) circle (0.58cm);
\node(14) at (0,5) {$\begin{tikzpicture}[scale = 0.85]
\node (60) at (0,1) {\small{$1$}};
\node (61) at (0.5,0) {\small{$2$}};
\node (62) at (-0.5,0) {\small{$3$}};
\draw[->] (60) -- (61);
\draw[->] (61) -- (62);
\draw[->] (62) -- (60);
\end{tikzpicture} $};
\draw (0,4.95) circle (0.58cm);

\draw (7) to node[xshift = -5pt, yshift = -5pt, scale =0.9]{$\mu_3$} (1);
\draw (10) to node[xshift = -5pt, yshift = -5pt, scale =0.9]{$\mu_1$} (7);
\draw (10) to node[xshift = 7pt, scale =0.9]{$\mu_2$} (8);
\draw (10) to node[xshift = 5pt, yshift = -5pt, scale =0.9]{$\mu_3$} (9);
\draw (7) to node[xshift = 5pt, yshift = -5pt, scale =0.9]{$\mu_2$} (2);
\draw[blue] (8) to node[xshift = -5pt, yshift = -5pt, scale =0.9]{$\mu_3$} (3);
\draw[blue] (8) to node[xshift = 5pt, yshift = -5pt, scale =0.9]{$\mu_1$} (4);
\draw (9) to node[xshift = -5pt, yshift = -5pt, scale =0.9]{$\mu_2$} (5);
\draw (9) to node[xshift = 5pt, yshift = -5pt, scale =0.9]{$\mu_1$} (6);
\draw (1) to node[yshift = 5 pt, scale =0.9]{$\mu_1$} (2);
\draw[blue] (3) to node[yshift = 5 pt, scale =0.9]{$\mu_2$} (4);
\draw (5) to node[yshift = 5 pt, scale =0.9]{$\mu_3$} (6);
\draw (1) to node[xshift = -5pt, yshift = 5pt, scale =0.9]{$\mu_2$} (11);
\draw (2) to node[xshift = 5pt, yshift = 5pt, scale =0.9]{$\mu_3$} (11);
\draw[blue] (3) to node[xshift = -5pt, yshift = 5pt, scale =0.9]{$\mu_1$} (12);
\draw[blue] (4) to node[xshift = 5pt, yshift = 5pt, scale =0.9]{$\mu_3$} (12);
\draw (5) to node[xshift = -5pt, yshift = 5pt, scale = 0.9]{$\mu_1$} (13);
\draw (6) to node[xshift = 5pt, yshift = 5pt, scale = 0.9]{$\mu_2$} (13);
\draw (11) to node[xshift = -5pt, yshift = 5pt, scale = 0.9]{$\mu_1$} (14);
\draw (12) to node[xshift = 7pt, yshift = 5pt, scale = 0.9]{$\mu_2$} (14);
\draw (13) to node[xshift = 5pt, yshift = 5pt, scale = 0.9]{$\mu_3$} (14);
\end{tikzpicture}
\end{center}
\end{example}

\subsection{Properties of Toric Posets}

Throughout this paper, all of our quivers will be acyclic (no directed cycles). We will still specify the acyclic assumption for clarity throughout. Also, all of our quivers will be simple. By simple we mean that there are no parallel directed arcs; self-loops and anti-parallel directed arcs are already prevented due to the acyclic assumption.

For ordinary posets, we saw that the definition of the Hasse diagram and transitive closure  depended on chains in the poset. A similar story will be true for toric posets, except using \textit{toric chains}. We first define a \textit{toric directed path}.

\begin{definition}[\cite{toric}]  \rm
Elements $x_1,x_2,\ldots, x_{k-1}, x_k \in V$ form a {\it toric directed path} if $Q$ contains all of the following arcs:
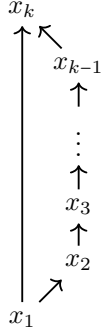
\begin{center}
\begin{tikzpicture}[scale = 0.75]
\node(1) at (0,0) {$x_1$};
\node(2) at (1,1) {$x_2$};
\node(3) at (1,2) {$x_3$};
\node(4) at (1,3.25){$\vdots$};
\node(5) at (1,4.5){$x_{k-1}$};
\node(6) at (0,5.5){$x_k$};
\draw[line width = 0.25 mm, ->] (1) -- (2);
\draw[line width = 0.25 mm, ->] (2) -- (3);
\draw[line width = 0.25 mm, ->] (3) -- (4);
\draw[line width = 0.25 mm, ->] (4) -- (5);
\draw[line width = 0.25 mm, ->] (5) -- (6);
\draw[line width = 0.25 mm, ->] (1) -- (6);
\end{tikzpicture}
\end{center}
\end{definition}
Let $C$ be the set of vertices in a toric directed path.
The \textit{length} of the toric directed path is $|C|-1$. We note that an arc is a toric directed path of length $1$ and a vertex is a toric directed path of length $0$.

\begin{definition}[\cite{toric}] \label{toric chain definition}  \rm
A \textit{toric chain} is a subset $V' \subseteq V$ that is totally ordered for every poset induced by an acyclic quiver in $[Q]$. 
\end{definition}

Just as chains are closed under subsets in a poset, toric chains are closed under subsets in a toric poset.
In \cite[Prop 6.3]{toric}, Develin, Macauley, and Reiner show that a subset $V' \subseteq V$ is a toric chain if and only if the elements of $V'$ lie along a toric directed path. Moreover, for $Q' \in [Q]$, every toric directed path in $Q'$ is contained in at least one maximal toric directed path, and every toric chain of $[Q]$ is contained in at least one maximal toric chain. By maximal, we mean with respect to containment.

\begin{definition} \label{def: toric comparability} \rm Two elements $a,b$ of a toric poset $[Q]$ are \textit{torically comparable} if there exists a toric chain in $[Q]$ that $a,b$ lie on together. If there is no toric chain that $a,b$ lie on together, we say that $a,b$ are \textit{torically incomparable}.
\end{definition}

Similarly to the ordinary poset setting, a toric poset may arise as a chamber in the complement of the corresponding toric graphic hyperplane arrangement for several graphs. Thus, there is also ambiguity for toric posets when identifying $[Q]$ with an acyclic quiver. However, once again we have two natural choices which we will define below:
\begin{enumerate}[i.]
\item $\overline{[Q]}$, the {\it toric transitive closure} of $[Q]$, and
\item $[Q]_{\rm{Hasse}}$, the {\it toric Hasse diagram} corresponding to $[Q]$.
\end{enumerate}

\begin{definition}[\cite{toric}] \rm
Let $[Q]$ be a toric poset. To define the toric transitive closure of $[Q]$, we must first choose a representative $Q'\ \in [Q]$. The \textit{toric transitive closure} of $Q'$, denoted $\overline{Q'}$ is the quiver where one adds to the underlying graph of $Q'$ all edges $\{i,j\}$ if $i$ and $j$ live on a toric chain and directs $i \rightarrow j$ if there exists a toric directed path from $i$ to $j$ in $Q'$. Then, the \textit{toric transitive closure} of $[Q]$, denoted $\overline{[Q]}$ is defined as $\overline{[Q]}:= [\overline{Q'}]$.

In \cite[Cor. 7.3]{toric}, the authors show that the toric transitive closure $\overline{[Q]}$ does not depend on the choice of representative $Q' \in [Q]$.

In contrast to the toric transitive closure, we can define the toric analogue of a Hasse diagram as follows. Let $[Q]$ be a toric poset and choose a representative $Q' \in [Q]$. Let $Q'_{\rm{Hasse}}$ be the quiver constructed from $Q'$ by removing each edge $i\to j$ for which $Q'$ contains a toric directed path from $i$ to $j$ that is both non-maximal and has length strictly greater than one. The \emph{toric Hasse diagram} of $[Q]$, denoted $[Q]_{\rm{Hasse}}$, is defined as $[Q]_{\rm{Hasse}}:=[Q'_{\rm{Hasse}}]$.

In \cite[Cor. 9.2]{toric}, the authors show that the toric Hasse diagram does not depend on the choice of representative $Q' \in [Q]$.
Below we show one representative of $[Q]_{\rm{Hasse}}, [Q]$, and $\overline{[Q]}$. We label these quivers $Q_1,Q_2,Q_3$, respectively.
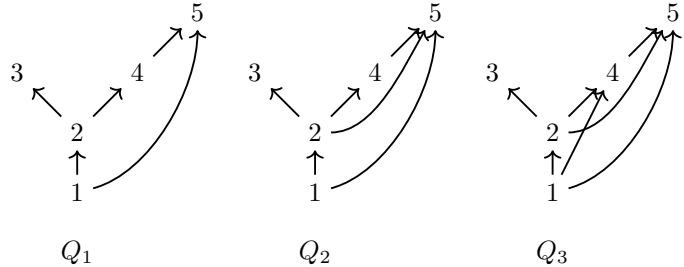
\begin{center}
\begin{tabular}{cccc}
\begin{tikzpicture}[scale = 0.8]
\node(1) at (0,0) {$1$};
\node(2) at (0,1) {$2$};
\node(3) at (-1,2) {$3$};
\node(4) at (1,2) {$4$};
\node(5) at (2,3) {$5$};
\node(6) at (0,-1) {$Q_1$};
\draw[line width = 0.25 mm, ->] (1) -- (2);
\draw[line width = 0.25 mm, ->] (2) -- (3);
\draw[line width = 0.25 mm, ->] (2) -- (4);
\draw[line width = 0.25 mm, ->] (4) -- (5);
\draw[line width = 0.25 mm, ->] (1) to [out = 15, in = 270, looseness = 0.8] (5);
\end{tikzpicture} & \begin{tikzpicture}[scale = 0.8]
\node(1) at (0,0) {$1$};
\node(2) at (0,1) {$2$};
\node(3) at (-1,2) {$3$};
\node(4) at (1,2) {$4$};
\node(5) at (2,3) {$5$};
\node(6) at (0,-1) {$Q_2$};
\draw[line width = 0.25 mm, ->] (1) -- (2);
\draw[line width = 0.25 mm, ->] (2) -- (3);
\draw[line width = 0.25 mm, ->] (2) -- (4);
\draw[line width = 0.25 mm, ->] (4) -- (5);
\draw[line width = 0.25 mm, ->] (1) to [out = 15, in = 270, looseness = 0.8] (5);
\draw[line width = 0.25 mm, ->] (2) to [out = 0, in = 240, looseness = 0.8] (5);
\end{tikzpicture} &
\begin{tikzpicture}[scale = 0.8]
\node(1) at (0,0) {$1$};
\node(2) at (0,1) {$2$};
\node(3) at (-1,2) {$3$};
\node(4) at (1,2) {$4$};
\node(5) at (2,3) {$5$};
\node(6) at (0,-1) {$Q_3$};
\draw[line width = 0.25 mm, ->] (1) -- (2);
\draw[line width = 0.25 mm, ->] (2) -- (3);
\draw[line width = 0.25 mm, ->] (2) -- (4);
\draw[line width = 0.25 mm, ->] (4) -- (5);
\draw[line width = 0.25 mm, ->] (1) -- (4);
\draw[line width = 0.25 mm, ->] (1) to [out = 15, in = 270, looseness = 0.8] (5);
\draw[line width = 0.25 mm, ->] (2) to [out = 0, in = 240, looseness = 0.8] (5);
\end{tikzpicture}
\end{tabular}
\end{center}
\end{definition}

We saw that Greene's rational function $\Psi^P(\mathbf{x})$
is defined as a sum of rational functions indexed by the set $\mathcal{L}(P)$ of linear extensions of a poset $P$. The toric analogue of Greene's rational function is defined as a sum of rational functions indexed by the set of \textit{toric total extensions} of a toric poset $[Q]$.

A \textit{toric total order} corresponds to a chamber in the complement of the toric complete graphic arrangement $\mathcal{A}_{tor}(K_V)$. A toric total order is of the form \begin{align*}
[w] := [(w_1,w_2, \ldots, w_n)]=\left\{ \right. &(w_1,w_2, \ldots, w_{n-1},w_n),\\
&(w_2,w_3,\ldots, w_n,w_1), \\
&\qquad \vdots \\
&\left. (w_n,w_1, w_2, \ldots, w_{n-1}) \right\}
\end{align*}
and we emphasize that a toric total order is a cyclic equivalence class.

\begin{definition}[\cite{toric}] \label{def: toric total extension} \rm
Let $[Q]$ be a toric poset and let $c$ be the toric chamber in the associated toric graphic hyperplane arrangement that corresponds to $[Q]$. A toric total order $[w]$ is a {\it toric total extension} of $[Q]$ if $c_{[w]} \subseteq c$. We denote the set of toric total extensions of $[Q]$ as $\torext$.
\end{definition}
Due to the following lemma, the set of toric total extensions of $[Q]_{\rm{Hasse}}$ is equal to the set of toric total extensions of $\overline{[Q]}$. Sometimes, it is more convenient to work in the toric transitive closure rather than the toric Hasse diagram and vice versa; E.g. the proof of Theorem \ref{thm:unified-recursion-steps} part (iii) uses the toric transitive closure.

\begin{lemma}
\label{lem: L-tor-independent-of-graph}
Let $[Q]$ be a toric poset. Then, $\mathcal{L}_{\rm{tor}}([Q]_{\rm{Hasse}}) = \torext=\mathcal{L}_{\rm{tor}}(\overline{[Q]})$.
\end{lemma}
\begin{proof}
Due to \cite[Cor. 7.3, Cor. 9.2]{toric}, it follows that $[Q]_{\rm{Hasse}}$ and $\overline{[Q]}$ correspond to the same toric chamber $c$ as $[Q]$, arising in different toric graphic hyperplane arrangements.
\end{proof}

In Remark \ref{Rmk: literature review}, we saw that sink-source flips are an instance of quiver mutation at a sink or source vertex. Thus, we 
carry over the notation for mutation and let $\mu_k(Q)$ denote the resulting quiver after flipping source (or sink) $k$ in $Q$ and we also let $[Q]_v$ denote the set of quivers in $[Q]$ where $v$ is a source.

\begin{prop}
\label{prop:[Q]_v-nonempty}
    For any vertex $v$ in any acyclic quiver $Q$, there exist an acyclic quiver $Q'$ in $[Q]$ having $v$ as a source. That is, $[Q]_v$ is never empty.
\end{prop}

\begin{proof}
Let $Q_1 \in [Q]$ and let $(a_1,a_2,\ldots,a_n) \in \mathcal{L}(Q_1)$ where $v=a_i$ for some $a_i$. If $i=1$, vertex $v$ is necessarily a source. Otherwise, we may flip $a_n$, since $a_n$ is a sink in $Q_1$. In the resulting quiver, vertex $a_n$ is a source and vertex $a_{n-1}$ is a sink. We can keep flipping in this way until we reach a quiver where $a_i$ is a source.
\end{proof}

\begin{prop} \label{prop: Toric lin ext reformulations} \rm For a toric poset $[Q]$, the set of toric total extensions can be written in terms of ordinary linear extensions in the following ways:

\begin{enumerate}[i.]
\item \label{L-tor-reformulation-1} $\torext = \{[w]: w \in \mathcal{L}(Q') \text{ for some } Q' \in [Q]\}$

\item \label{L-tor-reformulation-2} $\torext= \displaystyle \bigsqcup_{
Q' \in [Q]_1}
\left\{[1 \hat{w}]: \hat{w} \in \mathcal{L}(Q'-\{1\}) \right\}$
\end{enumerate}
where $\bigsqcup$ denotes disjoint union.

\end{prop}
\begin{proof}
\textbf{Assertion (i)}. We start by showing  \[\torext \subseteq \{[w]: w \in \mathcal{L}(Q') \text{ for some } Q' \in [Q]\}.\] Let $[\alpha] \in \torext$. Then by Definition \ref{def: toric total extension}, we have $c_{[\alpha]} \subseteq c_{[Q]}$. Let $\mathbf{x}=(x_1,x_2,\ldots,x_n)$ be a point in $c_{[\alpha]}$ and let $\mathbf{\tilde{x}} \in \mathbb{R}^n - \mathcal{A}(G)$ such that $\pi(\mathbf{\tilde{x}})=\mathbf{x}$. By \cite[Thm. 1.4]{toric}, $\mathbf{\tilde{x}} \in c_{w'}$ for some $w' \in [\alpha]$. Since $c_{[\alpha]} \subseteq c_{[Q]}$, we similarly have $\mathbf{\tilde{x}} \in c_{Q'}$ for some $Q' \in [Q]$. Thus, we have that $\mathbf{\tilde{x}} \in c_{Q'} \cap c_{w'}$. Using Equation \eqref{cones of braid arrangement}, it must be the case that $c_{w'} \subseteq c_{Q'}$, implying $w' \in \mathcal{L}(Q')$.

Now, we show \[\{[w]: w \in \mathcal{L}(Q') \text{ for some } Q' \in [Q]\} \subseteq \torext.\]
Consider $[\alpha]$ such that $\alpha \in \mathcal{L}(Q')$ for some $Q' \in [Q]$. If $\alpha = (a_1,a_2,\ldots, a_n)$, then $a_n$ is a sink in $Q'$, so we may flip at $a_n$. We note that $(a_n,a_1, \ldots, a_{n-1}) \in \mathcal{L}(\mu_{a_n}(Q'))$. We can keep flipping in this way until we have visited every element of $[\alpha]$. Then, by Equation \eqref{cones of braid arrangement}, we have that for all $w' \in [\alpha]$, $c_{w'} \subseteq c_{Q''}$ for some $Q'' \in [Q]$. From Theorem \cite[Thm. 1.4]{toric}, we have $\pi(c_{w'}) \subseteq c_{[\alpha]}$ for all $w' \in [\alpha]$ and $\pi(c_{Q''}) \subseteq c_{[Q]}$ for all $Q'' \in [Q]$. Since every point in $c_{[\alpha]}$ is the projection of a point in $c_{w'}$ for some $w' \in [\alpha]$, we have that $c_{[\alpha]} \subseteq c_{[Q]}$, and $[\alpha] \in \torext$. \\

\textbf{Assertion (ii)} We now show assertion (ii) is equivalent to assertion (i). If $[w] \in \torext$, then for all $w' \in [w]$, we have that $w' \in \mathcal{L}(Q')$ for some $Q' \in [Q]$. Moreover, there exists a unique $w' \in [w]$ such that $w'_1=1$. Thus, (i) is equivalent to
\[\bigcup_{Q' \in [Q]} \{[w]: w \in \mathcal{L}(Q'), w_1=1\},\]
which we can rewrite as 
\[\bigcup_{Q' \in [Q]} \{[1\hat{w}]: \hat{w} \in \mathcal{L}(Q'-\{1\})\}.\]
We show that this union is in fact disjoint. In other words, we show that if 
\[\{[1\hat{w}]:\hat{w} \in \mathcal{L}(Q'-\{1\})\} ~ \bigcap ~ \{[1\hat{w}]:\hat{w} \in \mathcal{L}(Q''-\{1\})\} \neq \emptyset\]
for $Q',Q'' \in [Q]$, then $Q'=Q''$. Consider $[1\hat{w}]$ such that $\hat{w} =(w_2,\ldots,w_n)  \in \mathcal{L}(Q'-\{1\}) \bigcap \mathcal{L}(Q''-\{1\})$. For all $\{i,j\} \in G - \{1\}$, one has 
\[i \rightarrow j \text{ in } Q'-\{1\} \iff \hat{w}^{-1}(i) < \hat{w}^{-1}(j)  \iff i \rightarrow j \text{ in } Q''-\{1\}.\] We also note that since quivers $Q',Q''$ have $1$ as a source, any edges incident to $1$ will be directed away from $1$. Thus, $Q' = Q''$. Therefore, we have shown that

\[\torext= \displaystyle \bigsqcup_{Q' \in [Q]_1}
\left\{[1 \hat{w}]: \hat{w} \in \mathcal{L}(Q'-\{1\}) \right\}.\qedhere \] 
\end{proof}

\begin{remark} \rm
Let $[Q]$ be a toric poset where $1$ is an isolated vertex in the underlying graph of $[Q]$. Then, $1$ is both a source and a sink, so when finding the set of toric total extensions of $[Q]$, we can still use part (ii) of Proposition \ref{prop: Toric lin ext reformulations}. 
\end{remark}
Before stating the next result, we recall that a \textit{bounded} poset $P$ is one that has a unique minimal element $\hat{0}$ and a unique maximal element $\hat{1}$.

\begin{prop}
\label{prop:Ltor-for-bounded-posets}
Let $P$ be a bounded poset, and let $Q$ be the quiver resulting from adding the directed edge $\hat{0} \rightarrow \hat{1}$ to the Hasse diagram $H(P)$. Then
one has a bijection 
$$
 \begin{array}{rcl}
 \theta: \mathcal{L}(P) &
\longrightarrow &\torext\\ 
(\hat{0},w_2,\ldots,w_{n-1},\hat{1}) &\longmapsto &[(\hat{0},w_2,\ldots,w_{n-1},\hat{1})].
\end{array}
$$
\end{prop}

\begin{proof}
We first show that this map is injective. Let $\mathbf{a}=(\hat{0},a_2, \ldots, a_{n-1},\hat{1})$ and $\mathbf{b} = (\hat{0}, b_2, \ldots, b_{n-1},\hat{1})$ be two linear extensions of $P$, and suppose $\theta(\mathbf{a})=\theta(\mathbf{b})$. Then, $[(\hat{0},a_2,\ldots, a_{n-1}, \hat{1})]=[(\hat{0},b_2,\ldots, b_{n-1}, \hat{1})]$. Since these two cyclic equivalence classes are equal, the representatives where $\hat{0}$ comes first are also equal.

We now show that $\theta$ is surjective. First, we consider $\torext$. Using Proposition \ref{prop: Toric lin ext reformulations} part (ii), we consider all quivers in $[Q]$ where $\hat{0}$ is a source and let $Q'$ be such a quiver. Since $Q'$ contains the directed edge $\hat{0} \rightarrow \hat{1}$, we cannot flip $\hat{1}$ or else $\hat{0}$ would no longer be a source. We note that besides $\hat{0}$ and $\hat{1}$ there are no other sinks nor sources in $Q'$ to consider flipping. Thus, $Q'$ is the only quiver in $[Q]$ with $\hat{0}$ a source, so an arbitrary toric total extension of $[Q]$ is of the form $[(\hat{0}, w_2, \ldots, w_{n-1}, \hat{1})]$. It is clear that $(\hat{0},w_2,\ldots, w_{n-1},\hat{1})$ is a linear extension of $P$ and thus we can conclude that $\Phi$ is surjective. 
\end{proof}

In 1991, Brightwell and Winkler showed that counting the number of linear extensions of an ordinary poset is a $\#P$-complete problem. The following result shows that counting the number of toric total extensions of a toric poset is also $\#P$-complete. Further discussion regarding $\#P$-completeness can be found in \cite{arora, brightwell, goldreich2008computational}.

\begin{theorem}    
\label{thm: sharp-P-complete}
Counting the toric total extensions for a toric poset $[Q]$ is $\#P$-complete.
\end{theorem}

\begin{proof}
This result follows from the main result of Brightwell and Winkler \cite{brightwell}, since Proposition \ref{prop:Ltor-for-bounded-posets} shows that counting the elements of $\torext$ has counting ordinary linear extensions $\mathcal{L}(P)$ as a special case. In particular, given a poset $P$, create a bounded poset $\hat{P}$ by adding $\hat{0},\hat{1}$ to $P$. The poset $\hat{P}$ has the same number of linear extensions as $P$ and then Proposition \ref{prop:Ltor-for-bounded-posets} produces a toric poset $[Q]$ that has exactly that many toric total extensions. 
\end{proof}

Although part (ii) of Proposition \ref{prop: Toric lin ext reformulations} provides a more efficient process for finding the set of toric total extensions relative to Proposition \ref{prop: Toric lin ext reformulations} part (i), we look for more efficient ways to compute this set. In Section ~\ref{sec: algorithm}, we provide a recursive algorithm to more efficiently compute the set of toric total extensions of a toric poset. However, as mentioned in the Introduction, since Theorem \ref{thm: sharp-P-complete} shows counting toric total extensions is a $\# P$-complete problem, one should not expect a very efficient algorithm for finding $\torext$.

\section[Properties of Greene's Toric Function]{Properties of $\psitor$} 

In this section, we show various identities regarding $\psitor$. We start by showing how the \textit{Kleiss-Kuijf relations} are a specific instance of Greene's theorem for strongly planar posets (recall Equation $\eqref{stronglyplanar}$).

\begin{prop} \label{bounded poset}
 Suppose $P$ is a bounded poset with minimal element $\hat{0}$ and maximal element $\hat{1}$. Let $Q$ be the quiver resulting from adding the directed edge $\hat{0} \rightarrow \hat{1}$ to the Hasse diagram $H(P)$. 
 Then, for the toric poset $[Q]$, we have 
 $$
 \psitor= 
 \frac{1} 
 {x_{\hat{1}}-x_{\hat{0}}} \Psi^{P}(\mathbf{x}).
 $$

\end{prop}

\begin{proof}
Using Definition \ref{def: toric Greene} and the bijection $\theta:  \mathcal{L}(P) \rightarrow \torext$
from Proposition \ref{prop:Ltor-for-bounded-posets}, one has 
\begin{align*}
\psitor
&= 
\sum_{[u] \in \torext} \Psi^{[u]}_{\rm{tor}}(\mathbf{x})\\
&=  \sum_{w \in \mathcal{L}(P)} 
\frac{1}{
(x_{\hat{0}}-x_{w_2})
(x_{w_2}-x_{w_3})
\cdots 
(x_{w_{n-2}}-x_{w_{n-1}})
(x_{w_{n-1}}-x_{\hat{1}}) 
(x_{\hat{1}}-x_{\hat{0}})}\\
&= \sum_{w \in \mathcal{L}(P)} 
\Psi^{w}(\mathbf{x})
\frac{1}{(x_{\hat{1}}-x_{\hat{0}})}
= \frac{1}{(x_{\hat{1}}-x_{\hat{0}})}
\Psi^P(\mathbf{x}).\qedhere
\end{align*}
\end{proof}

\begin{cor} \label{Bounded, planar} Let $P$ be a bounded, strongly planar poset with minimal and maximal elements $\hat{0}, \hat{1}$.
Let $\Delta$ be the set of bounded regions of $P$, and let $Q$ be the quiver resulting from adding the directed edge $\hat{0} \rightarrow \hat{1}$ in $H(P)$. Then, by Proposition \ref{bounded poset} and Equation \eqref{stronglyplanar}, we have
\[\psitor= \frac{1} 
 {x_{\hat{1}}-x_{\hat{0}}} 
 \frac{\prod_{\delta \in \Delta} (x_{\rm{min(\delta)}}-x_{\rm{max(\delta)}})}{\prod_{i \lessdot_{P} j} (x_i-x_j)}.
\]
\end{cor}

\begin{example} \rm Recall the strongly planar poset $P_2$ from Table \ref{table:1},
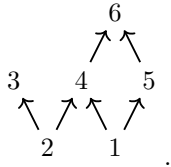
\begin{center}
\begin{tikzpicture}[scale = 0.9]
\node(1) at (-1,0) {2};
\node(2) at (0,0) {1};
\node(3) at (-1.5,1) {3};
\node(4) at (-0.5,1) {4};
\node(5) at (0.5,1) {5};
\node(6) at (0,2) {6};
\draw[line width = 0.25 mm,->](1) -- (3);
\draw[line width = 0.25 mm,->](1) -- (4);
\draw[line width = 0.25 mm,->](2) -- (4);
\draw[line width = 0.25 mm,->](2) -- (5);
\draw[line width = 0.25 mm,->](4) -- (6);
\draw[line width = 0.25 mm,->](5) -- (6);
\end{tikzpicture}.
\end{center}
We first adjoin $\hat{0}$ and $\hat{1}$ to $P_2$ in order to get a bounded poset $P_2'$. Then, we let $Q$ be the quiver resulting from adding the directed edge $\hat{0} \rightarrow \hat{1}$ to $H(P_2')$.
\pagebreak
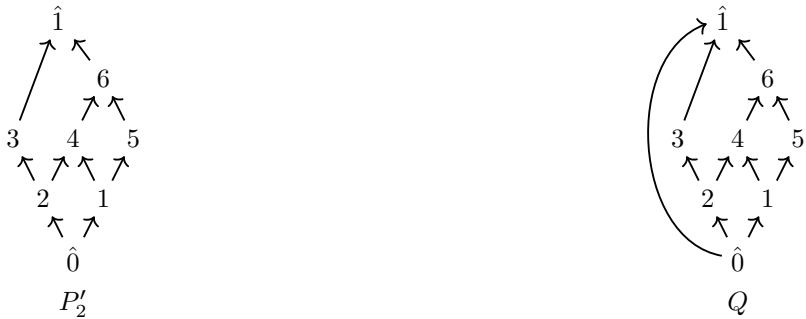
\begin{multicols}{2}
\begin{figure}[H]
\centering
\begin{tikzpicture}[scale = 0.8]
\node(2) at (-1,0) {2};
\node(1) at (0,0) {1};
\node(3) at (-1.5,1) {3};
\node(4) at (-0.5,1) {4};
\node(5) at (0.5,1) {5};
\node(6) at (0,2) {6};
\node(7) at (-.5,-1) {$\hat{0}$};
\node(8) at (-0.75,3) {$\hat{1}$};
\node(9) at (-0.5,-1.75) {$P_2'$};
\draw[line width = 0.25 mm,->](2) -- (3);
\draw[line width = 0.25 mm,->](2) -- (4);
\draw[line width = 0.25 mm,->](1) -- (4);
\draw[line width = 0.25 mm,->](1) -- (5);
\draw[line width = 0.25 mm,->](4) -- (6);
\draw[line width = 0.25 mm,->](5) -- (6);
\draw[line width = 0.25 mm,->](7) -- (2);
\draw[line width = 0.25 mm,->](7) -- (1);
\draw[line width = 0.25 mm,->](3) -- (8);
\draw[line width = 0.25 mm,->](6) -- (8);
\end{tikzpicture}
\end{figure} \columnbreak

\begin{figure}[H]
\centering
\begin{tikzpicture}[scale = 0.8]
\node(2) at (-1,0) {2};
\node(1) at (0,0) {1};
\node(3) at (-1.5,1) {3};
\node(4) at (-0.5,1) {4};
\node(5) at (0.5,1) {5};
\node(6) at (0,2) {6};
\node(7) at (-.5,-1) {$\hat{0}$};
\node(8) at (-0.75,3) {$\hat{1}$};
\node(9) at (-0.5,-1.75) {$Q$};
\draw[line width = 0.25 mm,->](2) -- (3);
\draw[line width = 0.25 mm,->](2) -- (4);
\draw[line width = 0.25 mm,->](1) -- (4);
\draw[line width = 0.25 mm,->](1) -- (5);
\draw[line width = 0.25 mm,->](4) -- (6);
\draw[line width = 0.25 mm,->](5) -- (6);
\draw[line width = 0.25 mm,->](7) -- (2);
\draw[line width = 0.25 mm,->](7) -- (1);
\draw[line width = 0.25 mm,->](3) -- (8);
\draw[line width = 0.25 mm,->](6) -- (8);
\draw[line width = 0.25 mm, ->] (7) to [out=170,in=200]  (8);
\end{tikzpicture}  
\end{figure}
\end{multicols}
\noindent Considering the toric poset $[Q]$, 
\[\psitor = \frac{(x_2-x_{\hat{1}})(x_1-x_6)(x_{\hat{0}}-x_4)}{(x_{\hat{1}}-x_{\hat{0}})(x_{\hat{0}}-x_2)(x_{\hat{0}}-x_1)(x_2-x_3)(x_2-x_4)(x_1-x_4)(x_1-x_5)(x_4-x_6)(x_5-x_6)(x_3-x_{\hat{1}})(x_6-x_{\hat{1}})}.\]
\end{example}

We emphasize that when finding $\psitor$, we look at the bounded regions of $P_2'$, not of $Q$. Before stating the next result, we first review the definition of a \textit{shuffle set}. Given two ordered sets $\mathbf{b} = (b_1, \ldots, b_k)$ and $\mathbf{c} = (c_1, \ldots, c_j)$, the shuffle set $\mathbf{b} \shuffle \mathbf{c}$ is the set of all permutations of $(b_1,\ldots,b_k,c_1,\ldots,c_j)$ such that the subsequences of the $b_i$ and $c_i$ appear in the same order as in $\mathbf{b}$ and $\mathbf{c}$ respectively.

\begin{example} \rm Consider $\mathbf{b}=(b_1,b_2)$ and $\mathbf{c}=(c_1,c_2)$. Then, the shuffle set $\mathbf{b} \shuffle \mathbf{c}$ is the following:
\[\mathbf{b} \shuffle \mathbf{c}=\{(b_1,b_2,c_1,c_2),(b_1,c_1,b_2,c_2), (b_1,c_1,c_2,b_2), (c_1,b_1,b_2,c_2),(c_1,b_1,c_2,b_2), (c_1,c_2,b_1,b_2) \}.\]
    
\end{example}

The following corollary is a special case of Corollary \ref{Bounded, planar} applied to the poset $P$ on the right of Figure \ref{fig: KK with two posets}. Let $\mathbf{b}=(b_1,b_2, \ldots, b_k)$ and let $\mathbf{c}= (c_1,c_2,\ldots, c_j)$. As convention, let $b_{k+1}=c_{j+1}= \hat{1}$ and $b_0=c_0=\hat{0}$.

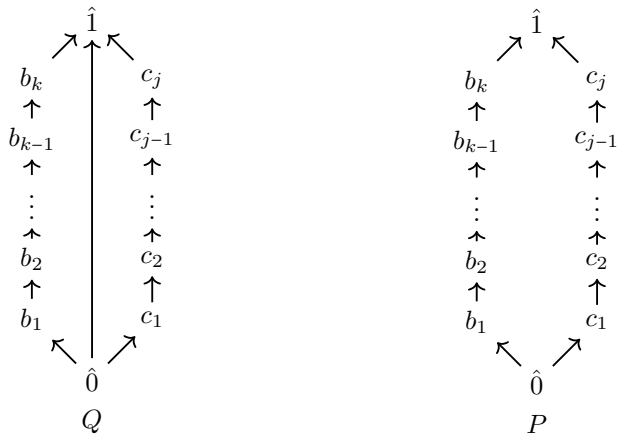
\begin{figure}[H]
\centering
\begin{tabular}{cc}
   \begin{tikzpicture}[scale = 0.8]
\node(2) at (0,0) {$\hat{0}$};
\node(3) at (-1,1) {$b_1$};
\node(4) at (-1,2) {$b_2$};
\node(5) at (-1,3) {$\vdots$};
\node(6) at (-1,4) {$b_{k-1}$};
\node(1) at (0,6) {$\hat{1}$};
\node(7) at (-1,5) {$b_k$};
\node(8) at (0,-0.7) {$Q$};
\draw[line width = 0.25 mm, ->] (2) -- (3);
\draw[line width = 0.25 mm, ->] (3) -- (4);
\draw[line width = 0.25 mm, ->] (4) -- (5);
\draw[line width = 0.25 mm, ->] (-1,3.4) -- (6);
\draw[line width = 0.25 mm, ->] (6) -- (7);
\draw[line width = 0.25 mm, ->] (7) -- (1);
\draw[line width = 0.25 mm, ->] (2) -- (1);
\node(8) at (1,1) {$c_1$};
\node(9) at (1,2) {$c_2$};
\node(10) at (1,3) {$\vdots$};
\node(11) at (1,4) {$c_{j-1}$};
\node(12) at (1,5) {$c_j$};
\draw[line width = 0.25 mm, ->] (2) -- (8);
\draw[line width = 0.25 mm, ->] (8) -- (9);
\draw[line width = 0.25 mm, ->] (9) -- (10);
\draw[line width = 0.25 mm, ->] (1,3.4) -- (11);
\draw[line width = 0.25 mm, ->] (11) -- (12);
\draw[line width = 0.25 mm, ->] (12) -- (1);
\end{tikzpicture} \hspace{3 cm} & 
   \begin{tikzpicture}[scale = 0.8]
\node(2) at (0,0) {$\hat{0}$};
\node(3) at (-1,1) {$b_1$};
\node(4) at (-1,2) {$b_2$};
\node(5) at (-1,3) {$\vdots$};
\node(6) at (-1,4) {$b_{k-1}$};
\node(1) at (0,6) {$\hat{1}$};
\node(7) at (-1,5) {$b_k$};
\node(8) at (0,-0.7) {$P$};
\draw[line width = 0.25 mm, ->] (2) -- (3);
\draw[line width = 0.25 mm, ->] (3) -- (4);
\draw[line width = 0.25 mm, ->] (4) -- (5);
\draw[line width = 0.25 mm, ->] (-1,3.4) -- (6);
\draw[line width = 0.25 mm, ->] (6) -- (7);
\draw[line width = 0.25 mm, ->] (7) -- (1);
\node(8) at (1,1) {$c_1$};
\node(9) at (1,2) {$c_2$};
\node(10) at (1,3) {$\vdots$};
\node(11) at (1,4) {$c_{j-1}$};
\node(12) at (1,5) {$c_j$};
\draw[line width = 0.25 mm, ->] (2) -- (8);
\draw[line width = 0.25 mm, ->] (8) -- (9);
\draw[line width = 0.25 mm, ->] (9) -- (10);
\draw[line width = 0.25 mm, ->] (1,3.4) -- (11);
\draw[line width = 0.25 mm, ->] (11) -- (12);
\draw[line width = 0.25 mm, ->] (12) -- (1);
\end{tikzpicture} 
\end{tabular}
    \caption{On the left we have one representative of the toric poset $[Q]$, which is the result of applying the procedure from Corollary \ref{Bounded, planar} to the poset $P$ on the right.}
     \label{fig: KK with two posets}
\end{figure}
\noindent 
{\bf Corollary \ref{cor: Kleiss-Kuijf relations}} \textit{({Kleiss-Kuijf Shuffle Relations})}
{\it For $\psitor$ where $[Q]$ is the toric poset shown in Figure \ref{fig: KK with two posets}, 

\begin{equation} \label{eq: KK relations toric poset way}
\psitor = \frac{(-1)^k}{\displaystyle {\prod_{r=0}^k (x_{b_{r+1}}-x_{b_r}) \cdot \prod_{s=0}^j (x_{c_s}-x_{c_{s+1}}) }},
\end{equation}
or equivalently,

\begin{equation} \label{eq: KK relations physics way}
\sum_{\mathbf{a} \in \mathbf{b} \shuffle  \mathbf{c}} \Psi_{\rm{tor}}^{[(\hat{1},\hat{0},\mathbf{a})]}(\mathbf{x}) = (-1)^k \Psi_{\rm{tor}}^{[(\hat{1},\rm{rev}(\mathbf{b}),\hat{0},\mathbf{c})]}(\mathbf{x}).
\end{equation}}

\begin{proof}
Corollary \ref{Bounded, planar} applied to $[Q]$ asserts

\begin{align}
\psitor &= \frac{1}{x_{\hat{1}}-x_{\hat{0}}} \cdot \frac{x_{\hat{0}}-x_{\hat{1}}}{\displaystyle{\prod^{k}_{r=0} (x_{b_r}-x_{b_{r+1}}) \cdot \prod^{j}_{s=0} (x_{c_s}-x_{c_{s+1}})}} \nonumber  \\&= \frac{(-1)^k}{\displaystyle {\prod_{r=0}^k (x_{b_{r+1}}-x_{b_r}) \cdot \prod_{s=0}^j (x_{c_s}-x_{c_{s+1}})}}.  \label{eqn: KK applied to [Q]}
\end{align} 

We now show Equation \eqref{eq: KK relations physics way}. By Proposition \ref{prop:Ltor-for-bounded-posets}, to find $\torext$, it suffices to find all ordinary linear extensions of $P$, which biject
with the linear extensions of the disjoint union of chains $b_1 < \cdots < b_k$ and $c_1 < \cdots < c_j$. Thus,
\begin{align*}
\torext
&=
\{[(\hat{0},\mathbf{a}, \hat{1})]: \mathbf{a} \in \mathbf{b} \shuffle \mathbf{c}\} \\
&=\{[(\hat{1},\hat{0},\mathbf{a})]:
\mathbf{a} \in \mathbf{b} \shuffle \mathbf{c} \}.
\end{align*}
By Definition \ref{def: toric Greene}, we have
\[\psitor = \sum_{\mathbf{a} \in \mathbf{b} \shuffle  \mathbf{c}} \Psi_{\rm{tor}}^{[(\hat{1},\hat{0},\mathbf{a})]}(\mathbf{x}).
\]
Rewriting Equation 
$\eqref{eqn: KK applied to [Q]}$ as 
$(-1)^k \Psi_{\rm{tor}}^{[(\hat{1},\rm{rev}(\mathbf{b}),\hat{0},\mathbf{c})]}(\mathbf{x})$, it follows that 
\[
\sum_{\mathbf{a} \in \mathbf{b} \shuffle  \mathbf{c}} \Psi_{\rm{tor}}^{[(\hat{1},\hat{0},\mathbf{a})]}(\mathbf{x})  = (-1)^k \Psi_{\rm{tor}}^{[(\hat{1},\rm{rev}(\mathbf{b}),\hat{0},\mathbf{c})]}(\mathbf{x}).
\qedhere \]
\end{proof}

\label{sec: properties of toric analogue}
Properties of $\Psi^{P}(\mathbf{x})$ shown by Boussicault, F\'eray, Lascoux, and Reiner in \cite{Cones} as well as properties by Greene in \cite{Greene} serve as motivation for the next few analogous properties of $\psitor$. Recall that in \cite{Cones}, the authors show that a poset $P$ is disconnected if and only if $\Psi^P(\mathbf{x})=0$. We will present a sufficient condition for when $\psitor = 0$, but first present a computational lemma that will help in the proof of this result.  It also appeared recently as \cite[Prop. 7.17]{Williams}, with a different proof.

\begin{lemma} \label{Lemma: Inner sum = 0} 
Let $\mathbf{a}=(a_1,a_2\ldots, a_m)$ and $\mathbf{b} =(b_1,b_2, \ldots, b_n)$. Then,

 \[\sum_{\substack{\mathbf{c} \in \mathbf{a} \shuffle \mathbf{b}}} \Psi_{\rm{tor}}^{[(1,\mathbf{c})]}(\mathbf{x})=0.\]
\end{lemma}

\begin{proof}
Let $\mathbf{\hat{a}}=(a_1,a_2\ldots, a_{m-1})$, and $\mathbf{\hat{b}}=(b_1,b_2, \ldots, b_{n-1})$. One has the following equalities, justified below:
\begin{align*}
    \sum_{\substack{\mathbf{c} \in \mathbf{a} \shuffle \mathbf{b}}} \Psi_{\rm{tor}}^{[(1,\mathbf{c})]}(\mathbf{x}) &= \sum_{\substack{\mathbf{c'} \in \mathbf{\hat{a}} \shuffle \mathbf{b}}} \Psi_{\rm{tor}}^{[(1,\mathbf{c'},a_m)]}(\mathbf{x}) + \sum_{\substack{\mathbf{c''} \in \mathbf{a} \shuffle \mathbf{\hat{b}}}} \Psi_{\rm{tor}}^{[(1,\mathbf{c''},b_n)]}(\mathbf{x}) \notag \\
    &= \sum_{\substack{\mathbf{c'} \in \mathbf{\hat{a}} \shuffle \mathbf{b}}} \Psi_{\rm{tor}}^{[(a_m,1,\mathbf{c'})]}(\mathbf{x}) + \sum_{\substack{\mathbf{c''} \in \mathbf{a} \shuffle \mathbf{\hat{b}}}} \Psi_{\rm{tor}}^{[(b_n,1, \mathbf{c''})]}(\mathbf{x}) \notag \\
    & = (-1)^{m-1}\Psi_{\rm{tor}}^{[(\rm{rev}(\mathbf{a}), 1, \mathbf{b})]}(\mathbf{x}) + (-1)^{m}\Psi_{\rm{tor}}^{[(b_n,\rm{rev}(\mathbf{a}),1,\mathbf{\hat{b}})]}(\mathbf{x})  
    = 0.
\end{align*}

In the first equality, we partition the shuffle set into linear extensions that end in $a_m$ and those that end in $b_n$. The second to last equality follows from applying the Kleiss-Kuijf relations (Corollary \ref{cor: Kleiss-Kuijf relations}) to each sum. The last equality holds because $[(\rm{rev}(\mathbf{a}), 1, \mathbf{b})]$ and $[(b_n,\rm{rev}(\mathbf{a}),1,\mathbf{\hat{b}})]$ are cyclically equivalent.
\end{proof}

We now prove our next main result. Recall that a cut vertex is a vertex where if removed, the number of connected components of the graph increases.

\vskip.1in
\noindent {\bf Theorem \ref{thm: Cut vertex}}
{ \it Let $[Q]$ be a toric poset, and let $G$ be the underlying undirected graph of $[Q]$. If $G$ is either disconnected with at least three vertices or has a cut vertex, then $\Psi_{\rm{tor}}^{[Q]}(\mathbf{x})=0$.}

\begin{proof}
We first consider the case that $G$ is disconnected with at least three vertices. Then, we can partition the vertex set $V$ of $G$ into two disjoint nonempty sets $A, B$ such that there are no edges $\{a,b\}$ with $a \in A, b \in B$. Since $|V|>2$, at least one of these sets has $2$ vertices. Without loss of generality, assume vertex $1$ is in this set. We will call this set $A$. Then, using Proposition \ref{prop: Toric lin ext reformulations} part (ii) to find $\torext$,
we have 
\[\torext= \displaystyle \bigsqcup_{Q' \in [Q]_1} 
\left\{[1 \hat{w}]: \hat{w} \in \mathcal{L}(Q'-\{1\}) \right\}.\]
Moreover, we can note that for $Q' \in [Q]$ where $1$ is a source in  $Q'$, an ordinary linear extension of $Q'-\{1\}$ is a shuffle of a linear extension $\mathbf{a}$ of elements in $Q_1=Q'|_{A-\{1\}}$ and a linear extension $\mathbf{b}$ of elements in $Q_2 = Q'|_{B}$. Thus, to find $\Psi_{\mathrm{tor}}^{[Q]}(\mathbf{x})$ using Definition \ref{def: toric Greene}, we have
\begin{equation} \label{eq for cut vertex}
\psitor= \sum_{
Q' \in [Q]_1} \sum_{\substack{\mathbf{a} \in \mathcal{L}(Q_1) \\ \mathbf{b} \in \mathcal{L}(Q_2)}} \sum_{\mathbf{c} \in \mathbf{a} \shuffle \mathbf{b}} \Psi_{\rm{tor}}^{[(1,\mathbf{c})]}(\mathbf{x}).
\end{equation}
By Lemma \ref{Lemma: Inner sum = 0}, the inner sum $\sum_{\mathbf{c} \in \mathbf{a} \shuffle \mathbf{b}} \Psi_{\rm{tor}}^{[(1,\mathbf{c})]}(\mathbf{x})$ is $0$. 

Now, we consider the case where $G$ has a cut vertex. Without loss of generality, let the cut vertex be the element $1$. Since $1$ is a cut vertex, if we were to remove $1$, we can once again partition the vertex set $V$ of $G$ into two disjoint nonempty sets $A, B$ such that there are no edges $\{a,b\}$ with $a \in A, b \in B$. Thus, for $Q' \in [Q]$ where $1$ is a source in  $Q'$, an ordinary linear extension of $Q'-\{1\}$ is a shuffle of a linear extension $\mathbf{a}$ of elements in $Q_1=Q'|_{A-\{1\}}$ and a linear extension $\mathbf{b}$ of elements in $Q_2 = Q'|_{B}$. To find $\psitor$, we again have Equation \eqref{eq for cut vertex}, so applying Lemma \ref{Lemma: Inner sum = 0}, we have that $\psitor=0$.
\end{proof}

\begin{remark} \rm In Theorem~\ref{thm: Cut vertex}, we need to assume the toric poset $[Q]$ has at least three vertices since if $[Q]$ has exactly two vertices $1,2$ and no arcs, then
$$
\psitor=\frac{1}{(x_1-x_2)(x_2-x_1)}=\frac{-1}{(x_1-x_2)^2} \neq 0.
$$
\end{remark}

\begin{remark} \rm Theorem \ref{thm: Cut vertex} gives only a sufficient condition for the vanishing of $\psitor$. We depict below a quiver $Q$ whose toric poset $[Q]$ has $\psitor = 0$, but where the vanishing is not implied by Theorem \ref{thm: Cut vertex}.
\begin{center}
\begin{tikzpicture}[scale=0.9]
\node(1) at (0,0) {$1$};
\node(2) at (-1,1) {$2$};
\node(3) at (0,1) {$3$};
\node(4) at (1,1) {$4$};
\node(5) at (0,2){$5$};
\draw[line width = 0.25 mm, ->] (1) -- (2);
\draw[line width = 0.25 mm, ->] (1) -- (4);
\draw[line width = 0.25 mm, ->] (1) -- (3);
\draw[line width = 0.25 mm, ->] (3) -- (5);
\draw[line width = 0.25 mm, ->] (2) -- (5);
\draw[line width = 0.25 mm, ->] (4) -- (5);
\end{tikzpicture}
\end{center}
\end{remark}

For ordinary posets, Boussicault, F\'eray, Lascoux, and Reiner show that linear terms in the denominator of $\Psi^P(\mathbf{x})$ correspond to cover relations of $P$. 
\begin{theorem} \emph{(\cite[Cor. 5.2]{Cones})}
For a connected poset $P$, the minimal denominator of $\Psi^{P}(\mathbf{x})$ is $\prod_{i\lessdot_P j}(x_i-x_j)$.
\end{theorem}

For toric posets, we have the following result. 

\vspace{1 mm}
\noindent {\bf Theorem \ref{thm: tor denom}}
{\it For $[Q]$ a toric poset, $\psitor$ can always be expressed over the denominator of  \[\prod_{\{i,j\} \in [Q]_{\rm{Hasse}}}(x_i-x_j)\] where we take the product over all edges $\{i,j\}$ in $[Q]_{\rm{Hasse}}$.}

\begin{proof}
We first note that by Definition \ref{def: toric Greene}, the denominator of $\psitor$ can only contain factors of the form $(x_i-x_j)$. A linear factor $(x_i-x_j)$ will not appear in the denominator of $\psitor$ if the sum of all $\Psi_{\rm{tor}}^{[w]}(\mathbf{x})$ that have $(x_i-x_j)$ in its denominator can be rewritten without this linear factor. We show that if there is no edge between vertices $i$ and $j$ in the underlying graph of $[Q]_{\rm{Hasse}}$, then the linear factor $(x_i-x_j)$, up to sign, will not appear in the denominator of $\psitor$.

Recall that the edge $\{i,j\}$ is in $[Q]_{\rm{Hasse}}$ if and only if $i$ and $j$ live on a toric directed path of length one. Thus, there are two cases to consider for when an edge $\{i,j\}$ is not in $[Q]_{\rm{Hasse}}$:

\begin{enumerate}[i.]
\item Vertices $i,j$ lie on a toric directed path of length greater than $1$.
\item Vertices $i,j$ do not lie on a toric directed path, i.e. they are torically incomparable.
\end{enumerate}

Using Definition \ref{def: toric Greene}, we find $\torext$ in order to compute $\psitor$. By Proposition \ref{prop: Toric lin ext reformulations} part (i), to find $\torext$ we must first compute the set of ordinary linear extensions for each quiver in $[Q]$. We first consider case (i) where vertices $i,j$ lie on some toric directed path of length greater than $1$ in $[Q]$. Then, for a quiver $Q' \in [Q]$, there is a directed path between $i$ and $j$ having length greater than $1$. This quiver does not have any ordinary linear extensions where $i$ is adjacent to $j$. Thus, it will not contribute any toric total extensions where $i$ is adjacent to $j$, so a linear factor of $(x_i-x_j)$ will not appear in $\psitor$.

We now consider case (ii). For a quiver $Q' \in [Q]$, there is either a directed path between $i,j$ with length greater than $1$ or $i,j$ are ordinary incomparable elements in $Q'$. The former case is the same as case (i). Now consider the latter case. If  $i,j$ are ordinary incomparable for $Q' \in [Q]$, for any linear extension $(a_1,a_2,\ldots, a_{l-2},i,j,a_{l+1},\ldots, a_n)$, there exists a linear extension $(a_1,a_2,\ldots, a_{l-2},j,i,a_{l+1},\ldots, a_n)$. Therefore, by Proposition \ref{prop: Toric lin ext reformulations} part (i), the set $\torext$ contains the toric total extensions \[[(a_1,a_2,\ldots, a_{l-2},i,j,a_{l+1},\ldots, a_n)] \text{ and } [(a_1,a_2,\ldots, a_{l-2},j,i,a_{l+1},\ldots, a_n)].\]
All toric total extensions of $[Q]$ where $i$ is adjacent to $j$ pair up in this way. Thus, it suffices to show that the expression
\begin{equation}
\label{sum-without-pole}
\Psi^{[(a_1,a_2,\cdots, a_{l-2},i,j,a_{l+1},\ldots, a_n)]}_{\rm{tor}}(\mathbf{x})+ \Psi^{[(a_1,a_2,\ldots, a_{l-2},j,i,a_{l+1},\ldots, a_n)]}_{\rm{tor}}(\mathbf{x})
\end{equation}
results in a rational function that does not have a factor of $x_i-x_j$ in its denominator.

After pulling out a common denominator factor of $$\prod_{\substack{r=1,2,\ldots,n-1\\r \neq l-2,l-1,l}} (x_{a_r}-x_{a_{r+1}})
$$
from both terms of \eqref{sum-without-pole}, and abbreviating
$x:=x_{a_{l-2}}, y:=x_{a_{l+1}},
$
we check the following identity, showing the rational function on the left can be rewritten without any denominator factor of $x_i-x_j$:
$$
\frac{1}{(x-x_i)(x_i-x_j)(x_j-y)}
+ \frac{1}{(x-x_j)(x_j-x_i)(x_i-y)}
=\frac{x-y}{(x-x_i)(x-x_j)(x_i-y)(x_j-y)}.\qedhere
$$
\end{proof}

\begin{remark} \rm
We note that a similar proof strategy
to that of Theorem \ref{thm: tor denom} appears in the recent work of Parisi, Sherman-Bennett, Tessler, and Williams, in their proof of \cite[Lem. 5.7]{Williams}.
\end{remark}

\begin{remark} \label{Remark: tordenom not necessarily smallest} \rm
For $[Q]$ in Figure \ref{fig:KK toric poset family}, we emphasize that the minimal denominator of $\psitor$,
as computed in Corollary~\ref{cor: Kleiss-Kuijf relations}, 
does not contain the linear factor $(x_{\hat{0}}-x_{\hat{1}})$, even though the edge $\{ \hat{0}, \hat{1}\}$ does appear in $[Q]_{\rm{Hasse}}$. 
\end{remark}

\section{Tricolored Subdivisions and Partial Cyclic Orders}  \label{sec: relationship to tricolored subdivisions}
In this section, we discuss the relationship between our Theorem \ref{thm: Cut vertex}  for $\psitor$ and an identity for Parke-Taylor factors $\mathrm{PT}(w)$ shown by Parisi, Tessler, Sherman-Bennett, and Williams in \cite[Thm. 7.11]{Williams}. 

Recall Equation \eqref{PT(w)}, which states

\[
\mathrm{PT}(w) = \frac{1}{(x_{w_2}-x_{w_1})(x_{w_3}-x_{w_2})\cdots(x_{w_n}-x_{w_{n-1}})(x_{w_1}-x_{w_n})}.
\]
In order to discuss their results for $\mathrm{PT}(w)$, we first must discuss \textit{partial cyclic orders} as well as \textit{tricolored subdivisions}.

\begin{definition}[\cite{megiddo1976partial}] \label{Partial cyclic order definition} \rm
A \textit{partial cyclic order} on a set $V$ is a ternary relation $C \subseteq V^3$ such that for all distinct $a,b,c,d \in C$:
\begin{enumerate}
\item (Cyclicity) $(a,b,c) \in C \implies (c,a,b) \in C$
\item (Asymmetry) $(a,b,c) \in C \implies (c,b,a) \notin C$
\item (Transitivity) $(a,b,c) \in C$ and $(a,c,d) \in C \implies (a,b,d) \in C$.
\end{enumerate}
A partial cyclic order $C$ is a \textit{total cyclic order} if for all $a,b,c \in V$, either $(a,b,c) \in C$ or $(a,c,b) \in C$.
\end{definition}

\begin{definition}  \label{cyclic extension} \rm
A total cyclic order $C$ is a \textit{circular extension} of a cyclic order $C'$ if $C' \subseteq C$.    
\end{definition}

While total cyclic orders are equivalent to toric posets $[w]$ for total orders $w$ (and thus a cyclic extension of a partial cyclic order can be seen to be the same as a toric total extension of a toric poset), in general partial cyclic orders are not the same as toric posets.
For instance, in this section, we will discuss the set of total cyclic orders Ext($C$) of a partial cyclic order $C$, and illustrate how this set behaves differently from  $\torext$ for a toric poset $[Q]$. We also note that Ext($C$) can sometimes be empty, which never occurs for $\torext$. An example of a partial cyclic order with an empty set of total cyclic orders is shown in \cite[Ex. 5]{megiddo1976partial}.

In \cite{Williams}, the authors associate partial cyclic orders to tricolored subdivisions of polygons.

\begin{definition}[\cite{Williams}] \rm
Let $\mathbf{P}_n$ be a convex $n$-gon with vertices labeled from $1$ to $n$ in clockwise order. A \textit{tricolored subdivision} $\tau$ is a partition of $\mathbf{P}_n$ into black, white, and grey polygons such that all vertices of the polygons are vertices of $\mathbf{P}_n$ and two polygons sharing an edge have different colors.
\end{definition}

\begin{example} \rm In Figure \ref{fig:Tricolored subdivision}, we give two examples of tricolored subdivisions.

\label{Partial cyclic orders}
\begin{figure}[H] 
    \centering
\begin{tabular}{cc}
    \begin{tikzpicture}[scale = 0.5]
   \newdimen\R
   \R=2.7cm
   \draw (0:\R) \foreach \x in {60,120,...,360} {  -- (\x:\R) };
   \foreach \x/\l/\p in
     { 60/{4}/above,
      120/{3}/above,
      180/{2}/left,
      240/{1}/below,
      300/{6}/below,
      360/{5}/right
     }
     \node[inner sep=1pt,circle,draw,fill,label={\p:\l}] at (\x:\R) {};

      \filldraw[gray, opacity = 0.3] (0:\R) \foreach \x in {120,240,360} {  -- (\x:\R) };
   \foreach \x/\l/\p in
     {120/{3}/above,
      240/{1}/below,
      360/{5}/right}
     \node[inner sep=1pt,circle,draw,fill,label={\p:\l}] at (\x:\R) {};
     \draw(0:\R) \foreach \x in {120,240,360} {  -- (\x:\R) };
   \foreach \x/\l/\p in
     {120/{3}/above,
      240/{1}/below,
      360/{5}/right}
     \node[inner sep=1pt,circle,draw,fill,label={\p:\l}] at (\x:\R) {};
\end{tikzpicture} & \hspace{2 cm}
    \begin{tikzpicture}[scale = 0.5]
   \newdimen\R
   \R=2.7cm
   \draw (0:\R) \foreach \x in {60,120,...,360} {  -- (\x:\R) };
   \foreach \x/\l/\p in
     { 60/{4}/above,
      120/{3}/above,
      180/{2}/left,
      240/{1}/below,
      300/{6}/below,
      360/{5}/right
     }
     \node[inner sep=1pt,circle,draw,fill,label={\p:\l}] at (\x:\R) {};

      \filldraw[gray, opacity = 0.3] (0:\R) \foreach \x in {60,120,180,240} {  -- (\x:\R) };
   \foreach \x/\l/\p in
     {60/{4}/above,
      120/{3}/above,
      180/{2}/left,
      240/{1}/below}
     \node[inner sep=1pt,circle,draw,fill,label={\p:\l}] at (\x:\R) {};

     \draw(0:\R) \foreach \x in {60,240,360} {  -- (\x:\R) };
   \foreach \x/\l/\p in
     {60/{4}/above,
      240/{1}/below,
      360/{5}/right}
     \node[inner sep=1pt,circle,draw,fill,label={\p:\l}] at (\x:\R) {};

\filldraw[black, opacity = 0.75](0:\R) \foreach \x in {60,240,360} {  -- (\x:\R) };
   \foreach \x/\l/\p in
     {60/{4}/above,
      240/{1}/below,
      360/{5}/right}
     \node[inner sep=1pt,circle,draw,fill,label={\p:\l}] at (\x:\R) {};

\end{tikzpicture}
\end{tabular}
    \caption{Tricolored subdivisions}
    \label{fig:Tricolored subdivision}
\end{figure}
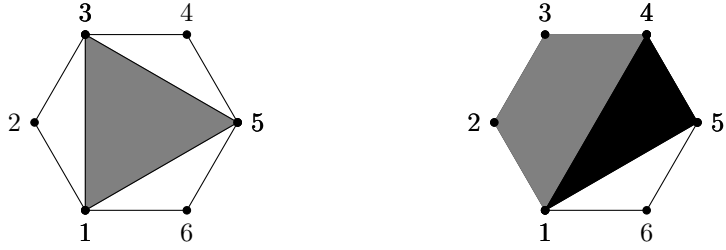
\end{example}

In \cite{Williams}, the authors associate a partial cyclic order to a tricolored subdivision
as follows.

\begin{definition}[\cite{Williams}] \rm \label {Partial cyclic order of tricolored subdivision} 
Let $\tau$ be a tricolored subdivision of $\mathbf{P}_n$. Consider only the black or white polygons in the subdivision. If a polygon $P_i$ in the subdivision is white, let $v_1, v_2, \ldots v_m$ be the clockwise order of its vertices; otherwise, let $v_1, v_2, \ldots v_m$ be the counterclockwise order of its vertices. Then, associate the cyclic chain $[(v_1,v_2, \ldots, v_m)]$ to $P_i$. The \textit{$\tau$-order} is the union of all of these cyclic chains, and this union is a partial cyclic order.
\end{definition}

\begin{example} \label{Cyclic extensions of tricolored subdivision} \rm Consider the tricolored subdivision on the left in Figure \ref{fig:Tricolored subdivision}. The three white triangles yield the cyclic chains $[(1,2,3)],[(1,5,6)],[(3,4,5)]$. In the tricolored subdivision on the right, the white triangle yields the cyclic chain $[(1,5,6)]$ and the black triangle yields $[(1,5,4)]$.
\end{example}

In \cite{Williams}, the authors prove the following theorem on cyclic extensions which arise from tricolored subdivisions. 

\begin{theorem}\emph{(\cite[Thm. 7.11]{Williams})}
\label{Tricolored subdivision theorem}
Let $\tau$ be a tricolored subdivision of $\mathbf{P}_n$ with at least one grey polygon, and let $C_{\tau}$ be the corresponding partial cyclic order. Then,
\[
\sum_{[w] \in \rm{Ext}(C_{\tau})} \mathrm{PT}(w)=0.
\]
\end{theorem}

Since cyclic extensions can be seen to be the same as toric total extensions, it is natural to wonder how Theorem \ref{Tricolored subdivision theorem} relates to our Theorem \ref{thm: Cut vertex}. We compare and contrast these theorems. In particular, we show that neither theorem implies the other, but acknowledge some overlap. We first show that Theorem  \ref{Tricolored subdivision theorem} does not imply Theorem \ref{thm: Cut vertex}. 
Revisiting Example \ref{Cyclic extensions of tricolored subdivision}, let $C_{\tau}$ be the partial cyclic order that is the union of the cyclic chains $[(1,2,3)],[(1,5,6)],[(3,4,5)]$. The set of cyclic extensions $\mathrm{Ext}(C_{\tau})$ contains the following cyclic extensions:

\vspace{0.25 cm}
\begin{center}
\begin{tabular}{ccccc}
   $[(1, 2, 3, 4, 5, 6)]$  & $[(1, 2, 4, 5, 3, 6)]$ & $[(1, 2, 4, 5, 6, 3)]$ & $[(1, 2, 5,3, 4, 6)]$ & $[(1, 2, 5, 3, 6, 4)]$ \\  $[(1, 2, 5, 6, 3, 4)]$ & 
      $[(1, 4, 2, 5, 3, 6)]$ & $[(1, 4, 2, 5, 6, 3)]$ & $[(1, 4, 5, 2, 3, 6)]$ &
      $[(1, 4, 5, 2, 6, 3)]$ \\  $[(1, 
  4, 5, 6, 2, 3)]$ & $[(1, 5, 2, 3, 4, 6)]$ &
  $[(1, 5, 2, 3, 6, 4)]$ & $[(1, 5, 2, 6,
   3, 4)]$ & $[(1, 5, 6, 2, 3, 4)]$
\end{tabular}
\end{center}
Then, by Theorem \ref{Tricolored subdivision theorem}, $\displaystyle \sum_{[w] \in \rm{Ext}(C_{\tau})} \mathrm{PT}(w)=0$. We now view the set of cyclic chains \[\{[(1,2,3)],[(1,5,6)],[(3,4,5)]\}\] 
as a set of toric chains (recall Definition \ref{toric chain definition}) and from this set, construct a toric poset $[Q]$. We draw one representative of $[Q]$ below.

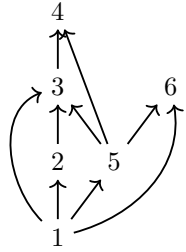
\begin{figure}[H]
\centering
\begin{tikzpicture}
\node(1) at (0,0) {1};
\node(2) at (0,1) {2};
\node(3) at (0,2) {3};
\node(4) at (0,3) {4};
\node(5) at (0.75,1) {5};
\node(6) at (1.5,2) {6};
\draw[line width = 0.25 mm,->](1) -- (2);
\draw[line width = 0.25 mm,->](2) -- (3);
\draw[line width = 0.25 mm,->](3) -- (4);
\draw[line width = 0.25 mm,->](1) -- (5);
\draw[line width = 0.25 mm,->](5) -- (3);
\draw[line width = 0.25 mm,->](5) -- (4);
\draw[line width = 0.25 mm,->](5) -- (6);
\draw[line width = 0.25 mm,->](1) to [out=135,in=200]  (3);
\draw[line width = 0.25 mm,->](1) to [out=15,in=280] (6);
\end{tikzpicture}
\end{figure} 
Finding the set of toric total extensions of $[Q]$ using Proposition \ref{prop: Toric lin ext reformulations} part (ii), we have \[\torext = \mathrm{Ext}(C_{\tau}) \backslash \{[1,2,3,4,5,6]\}.\] Thus, we have that 
\begin{align*}
\psitor
&=- \mathrm{PT}([1,2,3,4,5,6])+\sum_{[w] \in \mathrm{Ext}(C_{\tau})} 
\mathrm{PT}(w)\\
&=- \mathrm{PT}([1,2,3,4,5,6])\\
&=\frac{1}{(x_1-x_2)(x_2-x_3)(x_3-x_4)(x_4-x_5)(x_1-x_6)(x_5-x_6)} \neq 0.
\end{align*}

Now, we show an example where both theorems agree. Consider the tricolored subdivision on the right in Figure \ref{fig:Tricolored subdivision} and recall the partial cyclic order $C_{\tau}$ that is the union of cyclic chains $[(1,5,4)], [(1,5,6)]$. By Theorem \ref{Tricolored subdivision theorem}, $\sum_{[w] \in \rm{Ext}(C_{\tau})} \mathrm{PT}(w)=0$. Considering the set of cyclic chains as a set of toric chains, we construct the following representative of a toric poset:

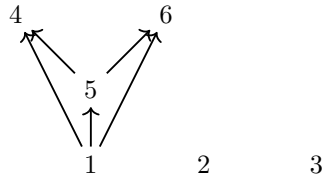
\begin{center}
\begin{tikzpicture} 
\node(1) at (0,0) {$1$};
\node(5) at (0,1) {$5$};
\node(4) at (-1,2) {$4$};
\node(6) at (1,2) {$6$};
\node(2) at (1.5,0) {$2$};
\node(3) at (3,0) {$3$};
\draw[line width = 0.25 mm, ->] (1) -- (5);
\draw[line width = 0.25 mm, ->] (5) -- (6);
\draw[line width = 0.25 mm, ->] (5) -- (4);
\draw[line width = 0.25 mm, ->] (1) -- (6);
\draw[line width = 0.25 mm, ->] (1) -- (4);
\end{tikzpicture}
\end{center}
By Theorem \ref{thm: Cut vertex}, for this toric poset $[Q]$, the rational function $\psitor$ is $0$.

We also note that there are toric posets that one can not construct using tricolored subdivisions. As an example, consider the following representative of a toric poset $[Q]$. 

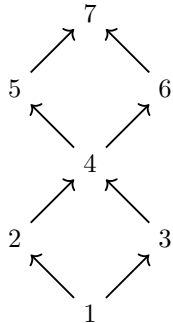
\begin{center}
\begin{tikzpicture} 
\node(1) at (0,0) {$1$};
\node(2) at (-1,1) {$2$};
\node(3) at (1,1) {$3$};
\node(4) at (0,2) {$4$};
\node(5) at (-1,3) {$5$};
\node(6) at (1,3) {$6$};
\node(7) at (0,4) {$7$};
\draw[line width = 0.25 mm, ->] (1) -- (2);
\draw[line width = 0.25 mm, ->] (1) -- (3);
\draw[line width = 0.25 mm, ->] (2) -- (4);
\draw[line width = 0.25 mm, ->] (3) -- (4);
\draw[line width = 0.25 mm, ->] (4) -- (5);
\draw[line width = 0.25 mm, ->] (4) -- (6);
\draw[line width = 0.25 mm, ->] (5) -- (7);
\draw[line width = 0.25 mm, ->] (6) -- (7);
\end{tikzpicture}
\end{center}
This toric poset has a cut vertex, so by Theorem \ref{thm: Cut vertex}, we have that $\psitor$ is $0$.  Since this toric poset does not have any toric chains with at least three elements, there is no tricolored subdivision of $\mathbf{P}_7$ that we could use to acquire $[Q]$.

\section{Source-Sink Equivalence with a Fixed Source} \label{Sec: source-sink equivalence}

In order to find the set of toric total extensions $\torext$ for a toric poset $[Q]$, we can use Proposition \ref{prop: Toric lin ext reformulations} part (ii). However, this requires one to find all quivers $Q' \in [Q]$ that have vertex $1$ as a source and there is currently no good algorithm for finding this set of quivers. Therefore, we are motivated to find methods that are more computationally efficient to compute $\torext$.

We provide a recurrence for finding $\torext$ (see Theorem \ref{thm:unified-recursion-steps}) that is similar to the recurrence for finding the set of ordinary linear extensions of posets seen in Lemma \ref{lem: basic poset fact}. To prepare, in this section, we prove Theorem \ref{thm: Source-sink equiv with fixed source}, which will be crucial in the proof of Theorem \ref{thm:unified-recursion-steps}.
 
\vspace{2 mm}
\noindent {\bf Theorem \ref{thm: Source-sink equiv with fixed source}}
{\it Let $v$ be any vertex in an acyclic quiver $Q$, and let $Q_1,Q_2$ be any two acyclic quivers in the subset $[Q]_v$ of the source-sink flip-equivalence class $[Q]$, so $v$ is a source in both $Q_1$ and $Q_2$. 

Then there exists a source-sink flip sequence from $Q_1$ to $Q_2$ such that every intermediate quiver in the sequence also has $v$ as a source. In other words, the flip sequence does not flip at $v$, nor at neighbors of $v$.}
\vspace{1 mm}

To prove Theorem \ref{thm: Source-sink equiv with fixed source}, we first show in Lemma \ref{lem: Flip seq. avoiding v} that we can freeze a vertex in the sense that we can find a flip sequence between any two representatives of a toric poset which never flips at this vertex. We highlight that the sequence described in Lemma \ref{lem: Flip seq. avoiding v} may flip at neighbors of $v$, just not at $v$ itself. We then relate this to quivers in $[Q]_v$ by building a quiver which reduces all vertices torically comparable with $v$, including $v$ itself, to one vertex and then in Lemma \ref{Lem: phi is a bijection}, showing that the flip equivalence class of this quiver is in bijection with $[Q]_v$.

\begin{lemma} \label{Lem: flips of a lin ext} Let $Q' \in [Q]$, and take an ordinary linear extension $(a_1,a_2, \ldots, a_n)$ of $Q'$. Performing a sequence of flips at the elements $a_1,a_2, \ldots, a_n$ in that order is well-defined as a sequence of flips, and results in the original quiver $Q'$.
\end{lemma}
\begin{proof}
We note that the process of flipping at vertices in this way is well-defined. Since $a_1$ is a minimal element of the linear extension, $a_1$ is a source in $Q'$, so we can flip at $a_1$. The resulting quiver will have $(a_2, \ldots, a_k, a_1)$ as a linear extension and $a_2$ as a source, so we can flip at $a_2$. Proceeding in this way, after flipping at the first $i-1$ vertices, we reach a quiver where we can flip at $a_i$. Flipping at all vertices once causes the orientation of each edge to change twice, resulting in the same quiver we started with, $Q'$.
\end{proof}

\begin{lemma} \label{lem: Flip seq. avoiding v} Let $Q',Q'' \in [Q]$ and let $v$ be a vertex in the underlying graph of $[Q]$. One can find a flip sequence from $Q'$ to $Q''$ that avoids flipping at $v$ (note we are allowed to flip at neighbors of v).
\end{lemma}

\begin{proof} Consider a flip sequence between $Q'$ and $Q''$. Assume at the $k^{th}$ step, vertex $v$ is flipped. Without loss of generality, assume at step $k-1$, vertex $v$ is a source in the quiver we will call $Q_{k-1}$ and at step $k$, the vertex $v$ is a sink in the quiver $Q_k$. Consider a linear extension of $Q_{k-1}$ where $v$ is first. Similar to Lemma \ref{Lem: flips of a lin ext}, we can follow the linear extension in reverse order, flipping all vertices except $v$. Thus, we've reached $Q_k$ without flipping at $v$.
\end{proof}

\begin{definition} \label{Def: definition of G^v} \rm
Consider a toric poset $[Q]$ and a quiver $Q_1 \in [Q]$. Let $G=(V,E)$ be the underlying graph of $[Q]$, and let $v \in V$.
We define a new graph $G^v = (V^v,E^v)$, which we call the \textit{$v$-incomparability graph of G}. For each $w \in V$ that is torically incomparable with $v$, we include a corresponding vertex $w^v \in V^v$. The vertex set $V^v$ consists of all such $w^v$ along with another vertex $v^*$. Roughly speaking, vertices in $V$ that are torically comparable with $v$ merge into a new vertex $v^*$ in $G^v$ and vertices that are torically incomparable with $v$ have a copy in $G^v$.

Before defining the edge set, we first define the distance between two vertices $a$ and $b$ in a quiver to be the length of a shortest path between them in the underlying undirected graph. The edge set $E^v$ of graph $G^v$ can be described as
\[
E^v = \{\{u^v,w^v\}: v^* \neq u^v,w^v \in V^v \text{ and } \{u,w\} \in E\} ~ \bigcup ~ \{\{v^*,w^v\} : \text{$w$ is distance 2 from $v$ in $\overline{Q_1}$\}.}
 \] 
Recall that the toric transitive closure $\overline{[Q]}$ of a toric poset $[Q]$ does not depend on the choice of representative in $[Q]$. Therefore, the choice of $Q_1 \in [Q]$ does not affect the set $E^v$.
\end{definition}

\begin{example} \rm We consider a representative $Q_1$ of a toric poset $[Q]$. Here $v:=1$. From Definition \ref{Def: definition of G^v}, vertices $1,2,3$ merge into $1^*$, since they are torically comparable with vertex $1$. Since $2$ merges into $1^*$ and there is an edge $\{2,4\}$ in the underlying graph of $[Q]$, we have an edge $\{1^*,4^1\}$ in $G^1$ and similarly for $\{1^*,5^1\}$. Edges $\{5^1,6^1\}$ and $\{4^1,6^1\}$ are copies of the edges in the underlying graph of $[Q]$.
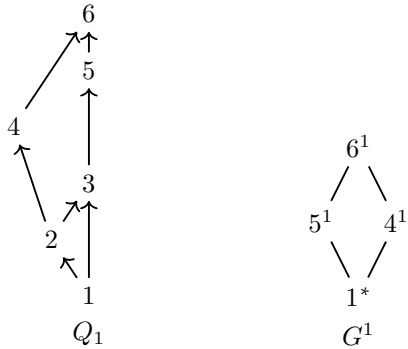
\begin{center}
\begin{tabular}{cc}
\begin{tikzpicture}
\node(1) at (0,0) {$1$};
\node(2) at (-0.5,0.75) {$2$};
\node(3) at (0,1.5) {$3$};
\node(4) at (-1,2.25) {$4$};
\node(5) at (0,3) {$5$};
\node(6) at (0,3.75) {$6$};
\node(7) at (0,-0.5) {$Q_1$};
\draw[line width = 0.25 mm, ->] (1) -- (2);
\draw[line width = 0.25 mm, ->] (2) -- (3);
\draw[line width = 0.25 mm, ->] (2) -- (4);
\draw[line width = 0.25 mm, ->] (4) -- (6);
\draw[line width = 0.25 mm, ->] (3) -- (5);
\draw[line width = 0.25 mm, ->] (5) -- (6);
\draw[line width = 0.25 mm, ->] (1) -- (3);
\end{tikzpicture} \hspace{2 cm}   &  \begin{tikzpicture}
\node(1) at (0,0) {$1^*$};
\node(2) at (0.5,1) {$4^1$};
\node(3) at (-0.5,1) {$5^1$};
\node(4) at (0,2) {$6^1$};
\node(5) at (0,-0.5) {$G^1$};
\draw[line width = 0.25 mm, -] (1) -- (2);
\draw[line width = 0.25 mm, -] (1) -- (3);
\draw[line width = 0.25 mm, -] (2) -- (4);
\draw[line width = 0.25 mm, -] (3) -- (4);
\end{tikzpicture} \\
\end{tabular}
\end{center} 
\end{example}
We now define a map:
 \[
 \Phi: [Q]_v \rightarrow \{\text{Orientations of $G^v$}\}.
 \]
For $Q' \in [Q]_v$, we define $\Phi(Q')$ below.

 \begin{enumerate}
     \item[i.] For $\{u^v,w^v\} \in E^v$ where $u^v,w^v \neq v^*$, if edge $\{u,w\} \in Q'$ is directed $u \rightarrow w$, direct $\{u^v,w^v\}$ as $u^v \rightarrow w^v$ in $\Phi(Q')$. Otherwise, direct $\{u^v,w^v\}$ as $w^v \rightarrow u^v$.
     \end{enumerate}
     
     \begin{enumerate}
     \item[ii.] If $\{v^*,u^v\} \in E^v$, then $v$ is distance 2 from $u$ in $\overline{Q'}$. Therefore, there exists a $w \in V$ such that there is an undirected path $(v,w,u)$ in $\overline{Q'}$. If the edge $\{w,u\}$ is directed $w \rightarrow u$ in $\overline{Q'}$, direct the edge $\{v^*,u^v\}$ as $v^* \rightarrow u^v$ in $\Phi(Q')$. Otherwise, direct $\{v^*,u^v\}$ as $u^v \rightarrow v^*$.
 \end{enumerate}
 In Example \ref{ex: example of phi}, we show three different acyclic quivers with $v:=1$ and show their images under $\Phi$.

There are four issues in the definition of $\Phi$ that we will address in succession:
\begin{itemize}
    \item The {\it choice of $w$} used in
    defining the direction  $u^v \rightarrow v^*$ will be immaterial (Proposition \ref{prop:orientation well defined}).
    \item The orientation $\Phi(Q')$ of $G^v$ really is {\it acyclic} (Proposition \ref{prop:no-directed-cycles}).
    \item The image of $[Q]_v$ under $\Phi$ lies {\it within a single $\equiv$-equivalence class} of acyclic orientations of $G^v$ (Lemma \ref{lem: orientations of G^v are flip equiv.}), and
    \item in fact, this image is an {\it entire} $\equiv$-equivalence class (Lemma \ref{Lem: phi is a bijection}).
\end{itemize}

 Given $u^v \in V^v$ such that $\{v^*,u^v\} \in E^v$, it is possible that there are multiple $w \in V$ such that there is an undirected path $(v,w,u)$ in $\overline{Q'}$. We show that the choice of $w$ used in condition (ii) will not affect the orientation that $\Phi$ imposes on the edge $\{v^*,u^v\}$.

 \begin{prop} \label{prop:orientation well defined}
The orientation assigned to $\{v^*,u^v\}$ is well defined.
\end{prop}
\begin{proof} Suppose that in $\overline{Q'}$, there exist two undirected paths between $v$ and $u$: $(v,w_1,u)$ and $(v,w_2,u)$, where $w_1,w_2 \in V$. Since $v$ is a source, edges $\{v,w_1\}, \{v,w_2\}$ are directed $v \rightarrow w_1$ and $v \rightarrow w_2$. Suppose for contradiction the edge $\{w_1,u\}$ has orientation $w_1 \rightarrow u$ and edge $\{w_2,u\}$ has orientation $u \rightarrow w_2$. Then, $\overline{Q'}$ contains the following toric chain.

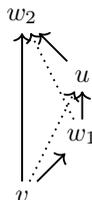
\begin{center}
\begin{tikzpicture}[scale = 0.8]
\node(1) at (0,0) {$v$};
\node(2) at (1,1) {$w_1$};
\node(3) at (1,2) {$u$};
\node(4) at (0,3) {$w_2$};
\draw[line width = 0.25 mm, ->] (1) -- (2);
\draw[line width = 0.25 mm, ->] (2) -- (3);
\draw[line width = 0.25 mm, ->] (3) -- (4);
\draw[line width = 0.25 mm, ->] (1) -- (4);
\draw[line width = 0.25 mm, dotted, ->] (1) -- (3);
\draw[line width = 0.25 mm, dotted, ->] (2) -- (4);
\end{tikzpicture}
\end{center}
Thus, edge $\{v,u\}$ exists in $\overline{Q'}$, which contradicts that $v$ and $u$ are distance 2 away in $\overline{Q'}$.
\end{proof}

\begin{example} \label{ex: example of phi} \rm We give a few representatives of toric posets and show their images under $\Phi$.
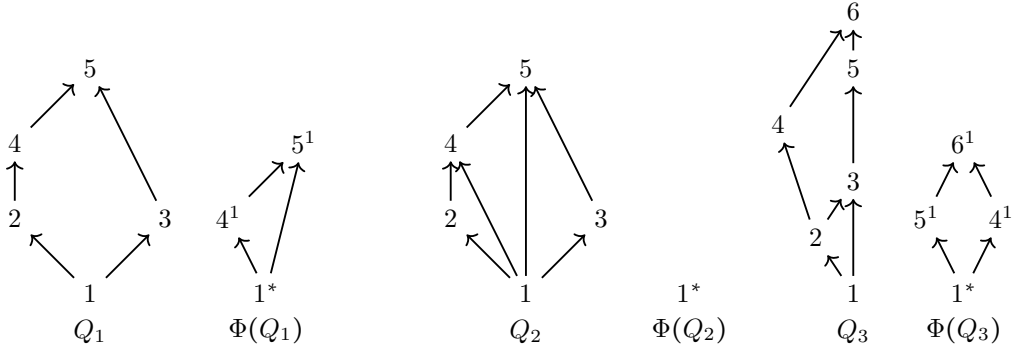
\begin{center}
\begin{tabular}{cccccc}
\begin{tikzpicture}
\node(1) at (0,0) {$1$};
\node(2) at (-1,1) {$2$};
\node(3) at (1,1) {$3$};
\node(4) at (-1,2) {$4$};
\node(5) at (0,3) {$5$};
\node(6) at (0,-0.5) {$Q_1$};
\draw[line width = 0.25 mm, ->] (1) -- (2);
\draw[line width = 0.25 mm, ->] (1) -- (3);
\draw[line width = 0.25 mm, ->] (2) -- (4);
\draw[line width = 0.25 mm, ->] (4) -- (5);
\draw[line width = 0.25 mm, ->] (3) to (5);
\end{tikzpicture} & 

\begin{tikzpicture}
\node(1) at (0,0) {$1^*$};
\node(4) at (-0.5,1) {$4^1$};
\node(5) at (0.5,2) {$5^1$};
\node(6) at (0,-0.5) {$\Phi(Q_1)$};
\draw[line width = 0.25 mm, ->] (1) -- (4);
\draw[line width = 0.25 mm, ->] (1) -- (5);
\draw[line width = 0.25 mm, ->] (4) -- (5);
\end{tikzpicture} \hspace{1 cm} &

\begin{tikzpicture}
\node(1) at (0,0) {$1$};
\node(2) at (-1,1) {$2$};
\node(4) at (-1,2) {$4$};
\node(3) at (1,1) {$3$};
\node(5) at (0,3) {$5$};
\node(6) at (0,-0.5) {$Q_2$};

\draw[line width = 0.25 mm, ->] (1) -- (2);
\draw[line width = 0.25 mm, ->] (1) -- (3);
\draw[line width = 0.25 mm, ->] (2) -- (4);
\draw[line width = 0.25 mm, ->] (4) -- (5);
\draw[line width = 0.25 mm, ->] (3) -- (5);
\draw[line width = 0.25 mm, ->] (1) -- (5);
\draw[line width = 0.25 mm, ->] (1) -- (4);
\end{tikzpicture}  &  

\begin{tikzpicture}
\node(1) at (0,0) {$1^*$};
\node(2) at (0,-0.5) {$\Phi(Q_2)$};
\end{tikzpicture} &

\begin{tikzpicture}
\node(1) at (0,0) {$1$};
\node(2) at (-0.5,0.75) {$2$};
\node(3) at (0,1.5) {$3$};
\node(4) at (-1,2.25) {$4$};
\node(5) at (0,3) {$5$};
\node(6) at (0,3.75) {$6$};
\node(7) at (0,-0.5) {$Q_3$};
\draw[line width = 0.25 mm, ->] (1) -- (2);
\draw[line width = 0.25 mm, ->] (2) -- (3);
\draw[line width = 0.25 mm, ->] (2) -- (4);
\draw[line width = 0.25 mm, ->] (4) -- (6);
\draw[line width = 0.25 mm, ->] (3) -- (5);
\draw[line width = 0.25 mm, ->] (5) -- (6);
\draw[line width = 0.25 mm, ->] (1) -- (3);
\end{tikzpicture} & 

\begin{tikzpicture}
\node(1) at (0,0) {$1^*$};
\node(2) at (0.5,1) {$4^1$};
\node(3) at (-0.5,1) {$5^1$};
\node(4) at (0,2) {$6^1$};
\node(5) at (0,-0.5) {$\Phi(Q_3)$};
\draw[line width = 0.25 mm, ->] (1) -- (2);
\draw[line width = 0.25 mm, ->] (1) -- (3);
\draw[line width = 0.25 mm, ->] (2) -- (4);
\draw[line width = 0.25 mm, ->] (3) -- (4);
\end{tikzpicture} \hspace{1 cm}
    
\end{tabular}
\end{center} 
\end{example}

\begin{prop}
\label{prop:no-directed-cycles}
The map $\Phi$ does not create any directed cycles.
\end{prop}
\begin{proof}
We show by contradiction that $\Phi(Q')$ does not contain any directed cycles. We consider two cases.

Case 1: Suppose that $\Phi(Q')$ contains a directed cycle that does not contain $v^*$. By definition of $\Phi$, the quiver $Q'$ must have contained a directed cycle, which contradicts that $Q'$ is an acyclic quiver.

Case 2: Suppose that $\Phi(Q')$ has a directed cycle that contains $v^*$, namely $(v^*,u^v_1,u^v_2,\ldots, u^v_k,v^*)$. Then, by definition of $E^v$, in $\overline{Q'}$ there exist vertices $w,x$ such that the edges $\{u_1,w\}$ and $\{u_k,x\}$ are directed as $w \rightarrow u_1$ and $u_k \rightarrow x$. In addition, edges $\{v,w\}$ and $\{v,x\}$ are directed $v \rightarrow w$ and $v \rightarrow x$, since $v$ is a source in $\overline{Q'}$. Then, in $\overline{Q'}$, we have the toric chain made up of directed paths $(v,w,u_1,\ldots, u_k,x)$ and $(v, x)$, implying that each $u_i$ is torically comparable with $v$. Thus, by definition of $V^v$, no such vertices $u^v_i$ would exist in $V^v$. Therefore, such a cycle in $\Phi(Q')$ cannot exist.
\end{proof}

Our next goal is to show that the image of $\Phi$ is a source-sink flip equivalence class of acyclic orientations of $G^v$. To show $\Phi(Q_1)$ and $\Phi(Q_2)$ are flip equivalent for $Q_1,Q_2 \in [Q]_v$, we first need terminology from Pretzel's work in \cite{pretzel} (also seen in \cite{pretzel2}). Let $Q$ be a quiver, and let $G$ be the underlying graph of $Q$. A \textit{walk} $W$ in $G$ is a sequence of vertices $v_1,v_2,\ldots, v_n$  such that $v_i$ is adjacent to $v_{i-1}$. The \textit{inverse walk} $-W$ is obtained by reversing the sequence $W$. If $v_1=v_n$, then the walk is called a \textit{circuit}. A circuit $C$ is \textit{trivial} if it is equal to traversing the walk $W$ and then the reverse walk $-W$. For a walk $W=(v_1,v_2,\ldots,v_k)$ in a quiver $Q$, an edge $\{v_{i-1},v_i\}$ is a \textit{forward edge} if it is directed towards $v_i$ in $Q$ and it is a \textit{backward edge} otherwise. Let $|C_Q^+|$ and $|C_Q^-|$ denote the number of forward edges and backward edges, respectively, of circuit $C$ in quiver $Q$. The flow-difference of $C$ in $Q$ is defined as
$d_Q(C):=|C_Q^+|-|C_Q^-|$. Two quivers $Q,Q'$ with the same underlying graph $G$ are said to have the same flow-difference if $d_Q(C)=d_{Q'}(C)$ for each circuit $C$ of $G$.

\begin{remark} \rm
The flow-difference of a circuit appears under other names in the literature. For instance, in \cite{Propp}, Propp refers to the flow-difference of a circuit as the circulation of the circuit and in \cite{toric}, the authors use Coleman's $\nu$-function on a circuit \cite{coleman}.
\end{remark}

Mosesian introduced the idea of pushing down maximal vertices in \cite{mosesian}. The operation of pushing down maximal vertices is the same as flipping sinks to sources. In 1984, Pretzel showed the following.
\begin{theorem} \emph{(\cite[Thm. 1']{pretzel})} \label{thm:Pretzel} Let $G$ be a finite simple graph. Two acyclic orientations of $G$ can be obtained from each other by pushing down if and only if they have the same flow-difference.
\end{theorem}

\begin{lemma} \label{lem: orientations of G^v are flip equiv.} 
Let $[Q]$ be a toric poset, and let $Q_1,Q_2 \in [Q]_v$. Then, $\Phi(Q_1)$ is flip equivalent to $\Phi(Q_2)$.
\end{lemma}

\begin{proof}
Using Theorem \ref{thm:Pretzel}, it suffices to show that $\Phi(Q_1)$ and $\Phi(Q_2)$ have the same flow-difference. Since the map $\Phi$ never creates any cycles, any circuits present in $\Phi(Q_1)$ and $\Phi(Q_2)$  must come from circuits in $Q_1$ and $Q_2$ respectively. There are three types of circuits in $Q_1$ and $Q_2$ to consider:
\begin{enumerate}[(i)]
\item Circuits that do not contain elements that are torically comparable to $v$.
\item Circuits that only contain elements that are torically comparable to $v$.
\item Circuits where a proper subset of the elements are torically comparable to $v$.
\end{enumerate}

In case (i), cycles that do not contain $v$ remain the same under $\Phi$. Since the flow-differences of $Q_1$ and $Q_2$ are equal in any such circuit, we immediately know that the flow-differences of $\Phi(Q_1)$ and $\Phi(Q_2)$ at such circuits are equal.

Circuits from case (ii) are deleted under $\Phi$, so they will not affect the flow-differences of $\Phi(Q_1)$ and $\Phi(Q_2)$.

We now consider case (iii), the most interesting case. Without loss of generality, let $C$ be a circuit in the underlying graph of $Q_1$ such that a nonempty proper subset of the vertices on $C$ are torically comparable to $v$ in $Q_1$. Such a circuit can be divided into walks $S_1,T_1,S_2,T_2,\ldots,S_m,T_m$ which appear consecutively where all vertices along $S_i$ are torically incomparable to $v$ and all vertices along $T_i$ are torically comparable to $v$. Under the map $\Phi$, all vertices on each $T_i$ are merged into $v^*$, edges that are completely contained in each $T_i$ are deleted, and each $S_i$ becomes a circuit that contains $v^*$.

We claim that in $Q_1$ and $Q_2$, the number of forward edges along each $S_i$ is fixed, which implies the number of backward edges is also fixed. Consider $S_1$ and let the first and last vertices be $w$ and $u$. By construction, $w$ is incident to $x$, a vertex torically comparable to $v$, and similarly $u$ is incident to $y$, a vertex that is torically comparable to $v$. Since $v$ is torically comparable to both $x$ and $y$, there exists a walk $W_1$ from $v$ to $x$ and a walk $W_2$ from $v$ to $y$. We note a circuit that is formed by these two walks and $S_1$: $W_1,S_1,W_2$. By Theorem \ref{thm:Pretzel}, this circuit has the same flow-difference in $Q_1$ and $Q_2$.  However, since $Q_1$ and $Q_2$ are in $[Q]_v$, the edges along the walks $W_1$ and $W_2$ must have all arrows directed away from $v$. Therefore, to preserve the flow-difference, the number of forward edges in $S_1$ must be the same in $Q_1$ and $Q_2$, implying that the flow-differences of the created circuit in $\Phi(Q_1)$ and $\Phi(Q_2)$ are equal.

We have now addressed all possible circuits in $\Phi(Q_1)$ and $\Phi(Q_2)$. Invoking Theorem \ref{thm:Pretzel} once more shows that $\Phi(Q_1)$ is flip equivalent to $\Phi(Q_2)$.
\end{proof}
 
 \begin{lemma} For vertex $v^*\neq k^v \in V^v$, the following diagram commutes
 \[
\begin{tikzcd}
{[Q]} \arrow[r,"\Phi"] \arrow[d,"\mu_k"] & {[\Phi(Q')]} \arrow[d,"\mu_{k^v}"]\\
{[Q]} \arrow[r,"\Phi"] & {[\Phi(Q')]}
\end{tikzcd}.
\]
 \end{lemma}
 \begin{proof}
We show that $\Phi(\mu_k(Q'))=\mu_{k^v}(\Phi(Q'))$.
Since $k^v \in V^v$, there is a corresponding vertex $k \in V$.
Suppose $k^v$ is a source in $\Phi(Q')$. Then, from the definition of $E^v$, vertex $k$ is also a source in $Q'$.

Choose an arbitrary edge $\{k,u\}$ in $Q'$. Note that $u \neq v$, since $v$ is a source. We have two cases:
 \begin{enumerate}
 \item $u$ is torically comparable with $v$ and thus $k^v \rightarrow v^*$ is a directed edge in $\Phi(Q')$ or
 \item $u$ is torically incomparable with $v$ and thus $k^v \rightarrow u^v$ is a directed edge in $\Phi(Q')$.
 \end{enumerate}
 Based on our two cases, flipping at $k^v$ in $\Phi(Q')$ yields either the edge $v^* \rightarrow k^v$ or $u^v \rightarrow k^v$. On the other hand, if we were to first flip at $k$ and then apply $\Phi$, edge $k \rightarrow u \in Q'$ again becomes  $v^* \rightarrow k^v$ or $u^v \rightarrow k^v$. The case where $k^v$ is a sink is similar.
 \end{proof}

We have now built up to our next key lemma that is used in the proof of Theorem \ref{thm: Source-sink equiv with fixed source}.

\begin{lemma} \label{Lem: phi is a bijection}  Let $[Q]$ be a toric poset, and let $Q' \in [Q]_v$. The map $\Phi$ is a bijection between $[Q]_v$ and $[\Phi(Q')]$.
\end{lemma}

\begin{proof} 
We first show $\Phi$ is injective. Suppose that $Q_1,Q_2 \in [Q]_v$ and that $\Phi(Q_1)=\Phi(Q_2).$ We have two types of edges in $\Phi(Q_1)$.
\begin{enumerate}
    \item Consider $\{u^v,w^v\} \in E^v$ such that $u^v, w^v \neq v^*$. From the definition of $E^v$, there is a corresponding edge $\{u,w\} \in E$. If edge $\{u^v,w^v\}$ has the same orientation in $\Phi(Q_1)$ and $\Phi(Q_2)$, then $\{u,w\}$ has the same orientation in $Q_1$ and $Q_2$.
    
    \item Consider $\{u^v,v^*\} \in E^v$. By definition of $E^v$, there exists at least one $y \in V$ such that there is a path $(v,y,u) \in \overline{Q'}$. Recall that by Proposition \ref{prop:orientation well defined}, the orientation of $\{y,u\}$ is the same for all such $y$. The orientation of $\{u^v,v^*\}$ comes from the orientation of any such edge $\{y,u\}$. If $\{u^v,v^*\}$ is oriented the same in $\Phi(Q_1)$ and $\Phi(Q_2)$, then all such edges $\{y,u\}$ are oriented the same in $Q_1$ and $Q_2$. 
\end{enumerate}\

The only edges in $Q_1$ and $Q_2$ that we have not yet accounted for are those lying on toric chains that contain $v$. Since $Q_1$ and $Q_2$ are in $[Q]_v$, we already know the orientation of all such edges. Thus, if $\Phi(Q_1)=\Phi(Q_2)$, then $Q_1=Q_2$.

Now we show that $\Phi$ is surjective. Let $Q_j \in [\Phi(Q')]$. By Lemma \ref{lem: Flip seq. avoiding v}, there is a flip sequence from $\Phi(Q')$ to $Q_j$, avoiding $v^*$. We emphasize that this flip sequence never flips $v^*$, although it might flip the neighbors of $v^*$ in $G^v$. By Lemma \ref{lem: orientations of G^v are flip equiv.}, if we apply this flip sequence to $Q'$ and then apply $\Phi$, we recover $Q_j$.
\end{proof}

\begin{proof}[Proof of Theorem \ref{thm: Source-sink equiv with fixed source}]
Let $Q_1, Q_2 \in [Q]_v$. In Lemma \ref{Lem: phi is a bijection}, we showed that $\Phi(Q_1)$ and $\Phi(Q_2)$ are flip equivalent. Moreover, using Lemma \ref{lem: Flip seq. avoiding v}, we can find a flip sequence between these that avoids flipping at $v^*$. Since $\Phi$ is a bijection, we can apply $\Phi^{-1}$ to each quiver in the sequence and after doing so, we have a sequence of quivers that are related by flips between $Q_1$ and $Q_2$. Since we do not flip at $v^*$, in the flip sequence in $[Q]$, we never flip at any vertex that is torically comparable to $v$. Therefore, the vertex $v$ remains a source at all intermediate quivers during the flip sequence between $Q_1$ and $Q_2$.
\end{proof}

One can view Theorem \ref{thm: Source-sink equiv with fixed source} simply as a statement regarding sink-source mutation of acyclic quivers. In Example \ref{ex: counterexamples to 1.10}, we show that Theorem \ref{thm: Source-sink equiv with fixed source} is not true if we drop the acyclic condition. 

\begin{example} \label{ex: counterexamples to 1.10} \rm Consider the following toric posets $[Q_1]$ and $[Q_2]$; we draw one representative of each below.

\begin{multicols}{2}
\begin{figure}[H]
\centering
\begin{tikzpicture}
\node(0) at (0,-0.5) {$Q_1$};
\node(1) at (0,0) {1};
\node(2) at (0,1) {2};
\node(3) at (0,2) {3};
\node(4) at (0,3) {4};
\node(5) at (1,3) {5};
\draw[line width = 0.25 mm, ->] (1) -- (2);
\draw[line width = 0.25 mm, ->] (2) -- (3);
\draw[line width = 0.25 mm, ->] (3) -- (4);
\draw[line width = 0.25 mm, ->] (4) -- (5);
\draw[line width = 0.25 mm, ->] (5) -- (3);
\end{tikzpicture}  
\end{figure} \columnbreak

\begin{figure}[H]
\centering
\begin{tikzpicture}
\node(0) at (0,-0.5) {$Q_2$};
\node(1) at (0,0) {1};
\node(2) at (-1,1) {2};
\node(3) at (1,1) {3};
\node(4) at (-1,2) {4};
\node(5) at (1,2) {5};
\node(6) at (0,3) {6};
\draw[line width = 0.25 mm, ->] (1) -- (2);
\draw[line width = 0.25 mm, ->] (1) -- (3);
\draw[line width = 0.25 mm, ->] (2) -- (4);
\draw[line width = 0.25 mm, ->] (2) -- (5);
\draw[line width = 0.25 mm, ->] (3) -- (4);
\draw[line width = 0.25 mm, ->] (3) -- (5);
\draw[line width = 0.25 mm, ->] (4) -- (5);
\draw[line width = 0.25 mm, ->] (5) -- (6);
\draw[line width = 0.25 mm, ->] (6) -- (4);
\end{tikzpicture}
\end{figure}
\end{multicols}

We first consider $Q_1$ and flip at vertex $1$ and then vertex $2$. For $Q_2$, we flip at vertices $1,2,3$, in that order. In both cases, the resulting quiver has vertex $1$ as a source, but there is no flip sequence between the resulting quiver and the starting quiver that avoids flipping at $1$ and at neighbors of $1$.
The author would like to thank Darij Grinberg and Scott Neville for providing the counterexamples involving $[Q_1]$ and $[Q_2]$, respectively.
\end{example}
\section{An Algorithm for Finding Toric Total Extensions} \label{sec: algorithm} We recall from Lemma \noindent \ref{lem: basic poset fact} the well-known recursive description of the set of linear extensions $\mathcal{L}(P)$ of an ordinary poset $P$.
\vspace{1 mm}

\noindent {\bf Lemma \ref{lem: basic poset fact}}
{\it Let $P$ be a poset, and let $a,b$ be two incomparable elements of $P$. Then, 
\[
\mathcal{L}(P) = \mathcal{L}(P_{a \rightarrow b}) \sqcup \mathcal{L}(P_{b \rightarrow a})
\]
where $P_{a \rightarrow b}$ is obtained from $P$ by adding the relation $a < b$ and $P_{b \rightarrow a}$ is defined similarly.}
\vspace{1mm}

 We use this fact as motivation to provide a recursive algorithm for finding the set of toric total extensions $\torext$ of a toric poset $[Q]$.

Let $Q$ be a quiver with vertices $a$ and $b$ such that neither arc $a \rightarrow b$ nor $b \rightarrow a$ is in $Q$.
We define $Q_{a \rightarrow b}$ to be the quiver $Q$ with an added arc $a \rightarrow b$, and $Q_{b \rightarrow a}$ is defined similarly. Note that an equivalent definition of $a,b$ being torically incomparable in $[Q]$ is that there exists a quiver $Q' \in [Q]$ such that $a,b$ are ordinary incomparable in $Q'$. 

\begin{theorem} 
\label{thm:unified-recursion-steps}
Let $a,b$ be two torically incomparable elements in the toric poset $[Q]$.
\begin{enumerate}
\item[(i)] If $a,b$ are in different connected components of the graph of $[Q]$, then $[Q_{a \rightarrow b}] = [Q_{b \rightarrow a}]$ and
\[ \torext = \mathcal{L}_{\rm{tor}}([Q_{a \rightarrow b}]) = \mathcal{L}_{\rm{tor}}([Q_{b \rightarrow a}]). \]
\item[(ii)] If $a,b$ are in the same connected component and $Q' \in [Q]$ is a representative where $a,b$ are ordinary incomparable, then the sets $\mathcal{L}_{\rm{tor}}([Q'_{a \rightarrow b}])$ and $\mathcal{L}_{\rm{tor}}([Q'_{b \rightarrow a}])$ are disjoint subsets of $\torext$, but the inclusion of the disjoint union $$\mathcal{L}_{\rm{tor}}([Q'_{a \rightarrow b}]) \sqcup \mathcal{L}_{\rm{tor}}([Q'_{b \rightarrow a}]) \subseteq \torext$$ may be proper.

\item[(iii)] On the other hand, assume that $a,b$ are distance two in the graph of the toric transitive closure $\overline{[Q]}$,
say both adjacent to the vertex $v$.  Then if one chooses $Q' \in \overline{[Q]}_v$, that is, $Q'$  is a representative of $\overline{[Q]}$ with  $v$ a source (as in  Proposition~\ref{prop:[Q]_v-nonempty}), the inclusion in (ii) becomes an equality: 
\[
\torext = \mathcal{L}_{\rm{tor}}([Q'_{a \rightarrow b}]) \sqcup \mathcal{L}_{\rm{tor}}([Q'_{b \rightarrow a}]).
\]
\end{enumerate}
\end{theorem}

\begin{proof}
\textbf{Assertion (i)} We note that adding either  the directed edge $a \rightarrow b$ to $Q$ or the directed edge $b \rightarrow a$ to $Q$ does not create any new cycles and therefore does not change the flow-difference of $Q$. By Theorem \ref{thm:Pretzel}, the toric poset $[Q_{a \rightarrow b}]$ is equal to $[Q_{b \rightarrow a}]$, so we have that $\mathcal{L}_{\rm{tor}}([Q_{a \rightarrow b}])=\mathcal{L}_{\rm{tor}}([Q_{b \rightarrow a}])$.

\textbf{Assertion (ii)} Suppose $a,b$ are in the same connected component of $G$, the underlying graph of $[Q]$. We first show that $\mathcal{L}_{\rm{tor}}([Q_{a \rightarrow b}])$ and $\mathcal{L}_{\rm{tor}}([Q_{b \rightarrow a}])$ are disjoint. The quivers $Q_{a \rightarrow b}$ and $Q_{b \rightarrow a}$ will have different flow-differences. By Theorem \ref{thm:Pretzel}, the toric posets $[Q_{a \rightarrow b}]$ and $[Q_{b \rightarrow a}]$ are not equal, so they correspond to disjoint toric chambers of $\mathcal{A}_{\rm{tor}}(G)$. Therefore, by Definition \ref{def: toric total extension}, $\mathcal{L}_{\rm{tor}}([Q_{a\rightarrow b}])$ and $\mathcal{L}_{\rm{tor}}([Q_{b \rightarrow a}])$ are disjoint. 

Let $c_{[Q]}$ be the toric chamber in the associated toric graphic hyperplance arrangement that corresponds to the toric poset $[Q]$ and let $c_{[Q_{a \rightarrow b}]}$ be defined similarly. By \cite[Prop. 3.2]{toric}, we have that $c_{[Q_{a \rightarrow b}]} \subseteq c_{[Q]}$ and $c_{[Q_{b \rightarrow a}]} \subseteq c_{[Q]}$. It follows that $\mathcal{L}_{\rm{tor}}([Q_{a \rightarrow b}]) \subseteq \torext$ and $\mathcal{L}_{\rm{tor}}([Q_{b \rightarrow a}]) \subseteq \torext$. Therefore, we have shown that $\mathcal{L}_{\rm{tor}}([Q_{a \rightarrow b}]) \bigsqcup \mathcal{L}_{\rm{tor}}([Q_{b \rightarrow a}]) \subseteq \torext.$ In Example \ref{Ex: proper torext}, we demonstrate an instance where this inclusion is in fact proper.

\textbf{Assertion (iii)}
Let $a,b$ be a torically incomparable pair such that both vertices are adjacent to a common vertex $v$ in the graph of $\overline{[Q]}$. We first show that in any representative $Q' \in \overline{[Q]}$ where $v$ is a source, elements $a,b$ are ordinary incomparable in $Q'$. Consider for the sake of contradiction that for such a representative $a$ and $b$ are ordinary comparable, i.e. up to relabeling, there is a directed path from $a$ to $b$.  

By assumption, vertices $a,b$ are connected to $v$ in the graph of $\overline{[Q]}$, so $a$ is torically comparable to $v$ and $b$ is torically comparable to $v$. Therefore, $a$ and $v$ both lie on a toric directed path and $b$ and $v$ live on a different toric directed path. Note that these toric directed paths cannot be the same, since by assumption $a$ and $b$ are torically incomparable. The most general case for these toric directed paths is shown here

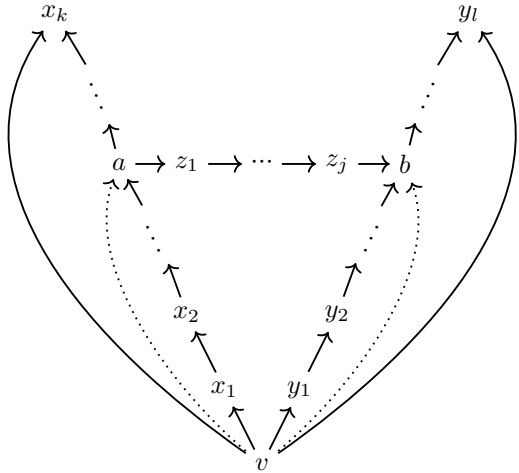
\begin{center}
  \begin{tikzpicture}
\node(0) at (0,0){$v$};
\node(1) at (-0.5,1){$x_1$};
\node(2) at (-1,2){$x_2$};
\node(15)[rotate=-28] at (-2.15,5) {$\ddots$};
\node(16)[rotate=28] at (2.15,5) {$\iddots$};
\node[rotate=-25](3) at (-1.4,3.1){$\ddots$};
\node(4) at (-1.9,4){$a$};
\node(5) at (-2.75,6){$x_k$};
\node(7) at (0.5,1){$y_1$};
\node(8) at (1,2){$y_2$};
\node(12) at (-1,4){$z_1$};
\node(13) at (0,4){$\cdots$};
\node(14) at (1,4){$z_j$};
\node[rotate=25](9) at (1.4,3.1){$\iddots$};
\node(10) at (1.9,4) {$b$};
\node(11) at (2.75,6) {$y_l$};
\draw[line width = 0.25 mm, ->] (0) -- (1);
\draw[line width = 0.25 mm, ->] (12) -- (13);
\draw[line width = 0.25 mm, ->] (13) -- (14);
\draw[line width = 0.25 mm, ->] (14) -- (10);
\draw[line width = 0.25 mm, ->] (4) -- (12);
\draw[dotted, line width = 0.25 mm, ->] (0) to [out=140,in=250, looseness = 0.8]  (4);
\draw[dotted, line width = 0.25 mm, ->] (0) to [out=40,in=290, looseness = 0.8]  (10);
\draw[line width = 0.25 mm, ->] (1) -- (2);
\draw[line width = 0.25 mm, ->] (2) -- (3);
\draw[line width = 0.25 mm, ->] (3) -- (4);
\draw[line width = 0.25 mm, ->] (4) -- (15);
\draw[line width = 0.25 mm, ->] (15) -- (5);
\draw[line width = 0.25 mm, ->] (0) to (7);
\draw[line width = 0.25 mm, ->] (7) -- (8);
\draw[line width = 0.25 mm, ->] (8) -- (9);
\draw[line width = 0.25 mm, ->] (9) -- (10);
\draw[line width = 0.25 mm, ->] (10) -- (16);
\draw[line width = 0.25 mm, ->] (16) -- (11);

\draw[line width = 0.25 mm, ->] (0) to [out = 35, in = 305]  (11);
\draw[line width = 0.25 mm, ->] (0) to [out = 145 , in =235] (5);
\end{tikzpicture} 
\end{center}
where the dashed edges from $v$ to $a$ and from $v$ to $b$ are two of the edges that appear in $\overline{[Q]}$, but are possibly not in $[Q]_{\rm{Hasse}}$.
 We see that the elements $v,x_1,x_2, \ldots, a,z_1, \ldots, z_j,b, \ldots, y_l$ form a toric chain, contradicting the fact that $a$ and $b$ are torically incomparable elements.

Using Proposition \ref{prop: Toric lin ext reformulations} part (ii), we have 
\[\mathcal{L}_{\rm{tor}}(\overline{[Q]}) = \bigsqcup_{
\substack{
\overline{Q'} \in \overline{[Q]}_v}}
\left\{[v\hat{w}]: \hat{w} \in \mathcal{L}(\overline{Q'}-\{v\}) \right\}.\] We have shown above that for any representative of $\overline{[Q]}$ where $v$ is a source,  $a$ and $b$ must be ordinary incomparable elements. Therefore, we can employ Lemma \ref{lem: basic poset fact} to each such representative:

\begin{align*}
\mathcal{L}_{\rm{tor}}(\overline{[Q]})  &= \mathlarger{\mathlarger{\bigsqcup}}_{\substack{
\overline{Q'} \in \overline{[Q]}_v}}
\left\{[v\hat{w}]: \hat{w} \in \mathcal{L}(\overline{Q'}_{a \rightarrow b}-\{v\}) \right\} \hspace{0.15 cm} \bigsqcup \hspace{0.15 cm}  \left\{[v\hat{w}]: \hat{w} \in \mathcal{L}(\overline{Q'}_{b \rightarrow a}-\{v\}) \right\}\\
&=\bigsqcup_{\substack{
\overline{Q'} \in \overline{[Q]}_v}}
\left\{[v\hat{w}]: \hat{w} \in \mathcal{L}(\overline{Q'}_{a \rightarrow b}-\{v\}) \right\}  \hspace{0.3 cm} \mathlarger{\mathlarger{\bigsqcup}}\hspace{0.3 cm}  \bigsqcup_{\substack{
\overline{Q'} \in \overline{[Q]}_v}}
\left\{[v\hat{w}]: \hat{w} \in \mathcal{L}(\overline{Q'}_{b \rightarrow a}-\{v\}) \right\}. \\
\end{align*}

We next claim that 
\begin{align*}
\bigsqcup_{\substack{
\overline{Q'} \in \overline{[Q]}_v}}
\left\{[v\hat{w}]: \hat{w} \in \mathcal{L}(\overline{Q'}_{a \rightarrow b}-\{v\}) \right\} 
&= \bigsqcup_{\substack{
\overline{Q'} \in [\overline{Q}_{a \rightarrow b}]}_v}
\left\{[v\hat{w}]: \hat{w} \in \mathcal{L}(\overline{Q'}-\{v\}) \right\}  \\
&= \bigsqcup_{\substack{
\overline{Q'} \in \overline{[Q_{a \rightarrow b}]}_v}}
\left\{[v\hat{w}]: \hat{w} \in \mathcal{L}(\overline{Q'}-\{v\}) \right\}
\end{align*}
and similarly, swapping the roles of $a$ and $b$. By Theorem \ref{thm: Source-sink equiv with fixed source}, there exists a flip sequence for any two quivers $Q',Q'' \in [Q]_v$ such that each intermediate quiver is also in $[Q]_v$. We emphasize that vertex $v$ is never flipped. Moreover, since vertices $a$ and $b$ lie on a toric directed path with $v$, both $a$ and $b$ also cannot be flipped. Therefore, every flip in this flip sequence commutes with the operation of adding the directed edge $a \rightarrow b$ (or $b \rightarrow a).$ We can use this special flip sequence between $Q'$ and $Q''$ to give a flip sequence between $Q'_{a \rightarrow b}$ and $Q''_{b \rightarrow a}$, and thus we have our first equality. The second equality follows from Lemma \ref{lem: L-tor-independent-of-graph}, so we can index our union as desired. We have that 

\begin{align*}
\mathcal{L}_{\rm{tor}}(\overline{[Q]}) &= \bigsqcup_{\substack{
\overline{Q'} \in \overline{[Q_{a \rightarrow b}]}_v}}
\left\{[v\hat{w}]: \hat{w} \in \mathcal{L}(\overline{Q'}-\{v\}) \right\} \hspace{0.3 cm}  \mathlarger{\mathlarger{\bigsqcup}} \hspace{0.3 cm}  \bigsqcup_{\substack{
\overline{Q'} \in \overline{[Q_{b \rightarrow a}]}_v}}
\left\{[v\hat{w}]: \hat{w} \in \mathcal{L}(\overline{Q'}-\{v\}) \right\}  \\
& = \mathcal{L}_{\rm{tor}}(\overline{[Q_{a \rightarrow b}]}) \hspace{0.3 cm}  \mathlarger{\mathlarger{\bigsqcup}} \hspace{0.3 cm}  \mathcal{L}_{\rm{tor}}(\overline{[Q_{b \rightarrow a}]}). \qedhere 
\end{align*} 
\end{proof}

Considering Theorem \ref{thm:unified-recursion-steps} part (ii), we now show an example where the inclusion of the disjoint union $\mathcal{L}_{\rm{tor}}([Q_{a \rightarrow b}]) \bigsqcup \mathcal{L}_{\rm{tor}}([Q_{b \rightarrow a}]) \subseteq \torext$ is proper.
\begin{example} \rm \label{Ex: proper torext}
We consider the toric posets $[Q], [Q_{4 \rightarrow 5}], [Q_{5 \rightarrow 4}],[Q_0]$ as well as their corresponding sets of toric total extensions. For each toric poset, we draw one representative below.
\begin{center}
  \begin{tabular}{cccc}
\begin{tikzpicture}
\node(0) at (0,-0.5) {$Q$};
\node(1) at (0,0) {$1$};
\node(2) at (-0.5,1) {$2$};
\node(3) at (0.5,1) {$3$};
\node(4) at (2.5,1) {$4$};
\node(5) at (0,2) {$5$};
\draw[line width = 0.25 mm, ->] (1) -- (2);
\draw[line width = 0.25 mm, ->] (1) -- (3);
\draw[line width = 0.25 mm, ->] (2) -- (5);
\draw[line width = 0.25 mm, ->] (3) -- (5);
\draw[line width = 0.25 mm, ->] (1) -- (4);
\end{tikzpicture} & \hspace{0.5 cm} \begin{tikzpicture}
\node(0) at (0,-0.5) {$Q_{4 \rightarrow 5}$};
\node(1) at (0,0) {$1$};
\node(2) at (-0.5,1) {$2$};
\node(3) at (0.5,1) {$3$};
\node(4) at (2.5,1) {$4$};
\node(5) at (0,2) {$5$};
\draw[line width = 0.25 mm, ->] (1) -- (2);
\draw[line width = 0.25 mm, ->] (1) -- (3);
\draw[line width = 0.25 mm, ->] (2) -- (5);
\draw[line width = 0.25 mm, ->] (3) -- (5);
\draw[line width = 0.25 mm, ->] (1) -- (4);
\draw[line width = 0.25 mm, ->] (4) -- (5);
\end{tikzpicture} \hspace{0.5 cm}  & \begin{tikzpicture}
\node(0) at (0,-0.5) {$Q_{5 \rightarrow 4}$};
\node(1) at (0,0) {$1$};
\node(2) at (-0.5,1) {$2$};
\node(3) at (0.5,1) {$3$};
\node(4) at (2.5,2.5) {$4$};
\node(5) at (0,2) {$5$};
\draw[line width = 0.25 mm, ->] (1) -- (2);
\draw[line width = 0.25 mm, ->] (1) -- (3);
\draw[line width = 0.25 mm, ->] (2) -- (5);
\draw[line width = 0.25 mm, ->] (3) -- (5);
\draw[line width = 0.25 mm, ->] (1)  to [out=15,in=250, looseness = 0.8]   (4);
\draw[ line width = 0.25 mm, ->] (5) -- (4);
\end{tikzpicture} &
\begin{tikzpicture}[scale=0.85]
\node(0) at (0,-0.5) {$Q_0$};
\node(1) at (0,0) {$1$};
\node(4) at (0,1) {$4$};
\node(5) at (0,2) {$5$};
\node(2) at (-1,3) {$2$};
\node(3) at (1,3) {$3$};
\draw[ line width = 0.25 mm, ->] (1) -- (4);
\draw[ line width = 0.25 mm, ->] (4) -- (5);
\draw[ line width = 0.25 mm, ->] (5) -- (2);
\draw[ line width = 0.25 mm, ->] (5) -- (3);
\draw[line width = 0.25 mm, ->] (1)  to [out=150,in=250, looseness = 0.8]   (2);
\draw[line width = 0.25 mm, ->] (1)  to [out=30,in=290, looseness = 0.8]   (3);
\end{tikzpicture}
 \end{tabular}
 \end{center}
 \begin{align*}
\mathcal{L}_{\rm{tor}}([Q_{4 \rightarrow 5}])  = \{&[(1,2,3,4,5)],[(1,2,4,3,5)], [(1,3,2,4,5)],[(1,3,4,2,5)][(1,4,2,3,5)],[(1,4,3,2,5)], \\&[(1,5,2,3,4)], [(1,5,2,4,3)],
[(1,5,3,2,4)],[(1,5,3,4,2)],[(1,5,4,2,3)],[(1,5,4,3,2)]\}  \\
\mathcal{L}_{\rm{tor}}([Q_{5 \rightarrow 4}])  = \{&[(1,2,3,5,4)],[(1,3,2,5,4)]\} \\
\mathcal{L}_{\rm{tor}}([Q_0])  = \{&[(1,4,5,2,3)],[(1,4,5,3,2)]\}
\end{align*}
For the set of toric total extensions of $[Q]$, we have
\begin{center}
$\torext =\mathcal{L}_{\rm{tor}}([Q_{4 \rightarrow 5}]) \bigsqcup \mathcal{L}_{\rm{tor}}([Q_{5 \rightarrow 4}]) \bigsqcup \mathcal{L}_{\rm{tor}}([Q_0])$.
\end{center}
\end{example}

Using Theorem \ref{thm:unified-recursion-steps}, we can recursively compute $\mathcal{L}_{\rm{tor}}(\overline{[Q]})$ in terms of
$\mathcal{L}_{\rm{tor}}(\overline{[Q_i]})$ for various toric
posets $[Q_i]$ each having more edges in their toric Hasse diagram than $[Q]$ has. Each of the latter toric posets has fewer toric total extensions, so they are easier to understand. We will show that when this iterative process ends, our resulting toric posets are exactly the toric total orders $[w]$; that is, $\mathcal{L}_{\rm{tor}}([w])=\{[w]\}$. Note that the transitive closure of a toric chain is a complete graph.

\begin{cor} \label{cor:algorithm} 
For a toric poset $[Q]$, iterative application of Theorem \ref{thm:unified-recursion-steps} gives us a finite algorithm for finding the set $\torext$, where the resulting toric posets correspond to toric total orders $[w]$. In other words, for every toric poset that is not a toric total order $[w]$, either Theorem \ref{thm:unified-recursion-steps} part (i) or part (iii) applies.

\end{cor}
\begin{proof}
Note that if the underlying graph of $\overline{[Q]}$ is not a complete graph, then it is disconnected or there exists a pair of elements with distance exactly 2 in the graph of $\overline{[Q]}$. If $\overline{[Q]}$ is disconnected, then we can apply Theorem \ref{thm:unified-recursion-steps} part (i). Otherwise, there exists two vertices with distance exactly $2$, so we can apply Theorem \ref{thm:unified-recursion-steps} part (iii). This process is finite, since there are a finite number of edges we can add to the graph before we have a complete graph. 
\end{proof}
\begin{example} \rm
In this example, we recalculate $\torext$ from Example \ref{ex: Calculating toric total extensions} (the toric poset is from Example \ref{Ex: toric poset example}), and 
we illustrate Theorem \ref{thm:unified-recursion-steps} with the following tree. 
The relationship between nodes and children of nodes in the tree is as follows: if node $\overline{[Q]}$ has children $\overline{[Q_i]}$, then $\mathcal{L}_{\rm{tor}}(\overline{[Q]})= \bigsqcup_{\overline{[Q_i]}} \mathcal{L}_{\rm{tor}}(\overline{[Q_i]})$. For each of our toric posets, we draw one representative. We note that in this example $\overline{[Q]}=[Q]_{\rm{Hasse}}$.

\begin{center}
\begin{tikzpicture}[scale = 0.85]
\node(0) at (0,-0.5){$Q$};
\node(1) at (0,0){$1$};
\node(2) at (-1,1){$2$};
\node(3) at (1,1){$3$};
\node(4) at (0,2){$4$};
\draw[line width = 0.25 mm, ->] (1) -- (2);
\draw[line width = 0.25 mm, ->] (2) -- (4);
\draw[line width = 0.25 mm, ->] (1)--(3);
\draw[line width = 0.25 mm, ->] (3) -- (4);

\node(0.1) at (-3.5,-2.5){$Q_{2 \rightarrow 3}$};
\node(6) at (-3.5,-2){$1$};
\node(7) at (-4.5,-1){$2$};
\node(8) at (-2.5,-1){$3$};
\node(9) at (-3.5,-0){$4$};
\draw[line width = 0.25 mm, ->] (6) -- (7);
\draw[line width = 0.25 mm, ->] (7) -- (9);
\draw[line width = 0.25 mm, ->] (6)--(8);
\draw[line width = 0.25 mm, ->] (8) -- (9);
\draw[line width = 0.25 mm, ->] (7) -- (8);

\node(0.2) at (3.5,-2.5){$Q_{3 \rightarrow 2}$};
\node(11) at (3.5,-2){$1$};
\node(12) at (2.5,-1){$2$};
\node(13) at (4.5,-1){$3$};
\node(14) at (3.5,0){$4$};
\draw[line width = 0.25 mm, ->] (11) -- (12);
\draw[line width = 0.25 mm, ->] (12) -- (14);
\draw[line width = 0.25 mm, ->] (11) -- (13);
\draw[line width = 0.25 mm, ->] (13) -- (14);
\draw[line width = 0.25 mm, ->] (13) -- (12);

\node(0.3) at (-5,-7){$(Q_{3 \rightarrow 2})_{1 \rightarrow 4}$};
\node(16) at (-5,-4.5){$1$};
\node(17) at (-5,-6.5){$2$};
\node(18) at (-5,-5.5){$3$};
\node(19) at (-5,-3.5){$4$};
\draw[line width = 0.25 mm, ->] (17) -- (18);
\draw[line width = 0.25 mm, ->] (18) -- (16);
\draw[line width = 0.25 mm, ->] (16) -- (19);
\draw[line width = 0.25 mm, ->] (17) to  [out=45,in=315, looseness = 0.8] (16);
\draw[line width = 0.25 mm, ->] (18) to [out=45,in=315, looseness = 0.8] (19);
\draw[line width = 0.25 mm, ->] (17) to [out=20,in=340] (19);

\node(0.4) at (-2,-7){$(Q_{2 \rightarrow 3})_{4 \rightarrow 1}$};
\node(21) at (-2,-3.5){$1$};
\node(22) at (-2,-4.5){$4$};
\node(23) at (-2,-5.5){$3$};
\node(24) at (-2,-6.5){$2$};
\draw[line width = 0.25 mm, ->] (24) -- (23);
\draw[line width = 0.25 mm, ->] (23) -- (22);
\draw[line width = 0.25 mm, ->] (22) -- (21);
\draw[line width = 0.25 mm, ->] (24) to  [out=45,in=315, looseness = 0.8] (22);
\draw[line width = 0.25 mm, ->] (23) to [out=45,in=315, looseness = 0.8] (21);
\draw[line width = 0.25 mm, ->] (24) to [out=20,in=340] (21);

\node(0.5) at (2,-7){$(Q_{3 \rightarrow 2})_{1 \rightarrow 4}$};
\node(21) at (2,-3.5){$4$};
\node(22) at (2,-4.5){$1$};
\node(23) at (2,-5.5){$2$};
\node(24) at (2,-6.5){$3$};
\draw[line width = 0.25 mm, ->] (24) -- (23);
\draw[line width = 0.25 mm, ->] (23) -- (22);
\draw[line width = 0.25 mm, ->] (22) -- (21);
\draw[line width = 0.25 mm, ->] (24) to  [out=45,in=315, looseness = 0.8] (22);
\draw[line width = 0.25 mm, ->] (23) to [out=45,in=315, looseness = 0.8] (21);
\draw[line width = 0.25 mm, ->] (24) to [out=20,in=340] (21);

\node(0.6) at (5,-7){$(Q_{3 \rightarrow 2})_{4 \rightarrow 1}$};
\node(25) at (5,-3.5){$1$};
\node(26) at (5,-4.5){$4$};
\node(27) at (5,-5.5){$2$};
\node(28) at (5,-6.5){$3$};
\draw[line width = 0.25 mm, ->] (28) -- (27);
\draw[line width = 0.25 mm, ->] (27) -- (26);
\draw[line width = 0.25 mm, ->] (26) -- (25);
\draw[line width = 0.25 mm, ->] (28) to  [out=45,in=315, looseness = 0.8] (26);
\draw[line width = 0.25 mm, ->] (27) to [out=45,in=315, looseness = 0.8] (25);
\draw[line width = 0.25 mm, ->] (28) to [out=20,in=340] (25);

\draw(-0.75, 0.25) -- (-2.5,-0.5);
\draw(0.75, 0.25) -- (2.5,-0.5);
\draw(0.75, 0.25) -- (2.5,-0.5);
\draw(-4.25, -1.75) -- (-4.75,-3);
\draw(-2.75, -1.75) -- (-2.25,-3);

\draw(2.75, -1.75) -- (2.25,-3);
\draw(4.25, -1.75) -- (4.75,-3);

\end{tikzpicture}
\end{center}
 Reading the toric total orders from left to right, we have $[(1,4,2,3)],[(1,2,3,4)],[(1,4,3,2)],[(1,3,2,4)]$, the set $\mathcal{L}_{\rm{tor}}(\overline{[Q]})$
found in Example \ref{ex: Calculating toric total extensions}.
\end{example}
\printbibliography

@article{Greene,
  title={A rational-function identity related to the {M}urnaghan--{N}akayama formula for the characters of ${S}_n$},
  author={Greene, C.},
  journal={Journal of Algebraic Combinatorics},
  volume={1},
  number={3},
  pages={235--255},
  year={1992},
  publisher={Springer}
}

@article{toric,
 author={Develin, M. and Macauley, M. and Reiner, V.},
  title={Toric partial orders},
  journal={Transactions of the American Mathematical Society},
  volume={368},
  number={4},
  pages={2263--2287},
  year={2016}
}

@article{Cones,
author={Boussicault, A. and F{\'e}ray, V. and Lascoux, A. and Reiner, V.},
  title={Linear extension sums as valuations of cones},
  journal={Journal of Algebraic Combinatorics},
  year={2010},
  volume = {35},
  pages={573--610}
}

@misc{Propp,
author={Propp, J.},
eprint ={math/0209005},
eprinttype ={arxiv},
title={Lattice structure for orientations of graphs},
year={2002}}

@misc{galashin2021poset,
 author={Galashin, P.},
 eprint ={2110.07257},
eprinttype = {arxiv},
title={Poset associahedra},
year={2021}}

@article{stanleyacyclic,
  title={Acyclic orientations of graphs},
  author={Stanley, R. P.},
  journal={Discrete Mathematics},
  volume={5},
  number={2},
  pages={171--178},
  year={1973},
  publisher={Elsevier}}

@article{Zaslavsky,
 author={Greene, C. and Zaslavsky, T.},
  title={On the interpretation of Whitney numbers through arrangements of hyperplanes, zonotopes, non-Radon partitions, and orientations of graphs},
  journal={Transactions of the American Mathematical Society},
  volume={280},
  number={1},
  pages={97--126},
  year={1983}}

@article{postnikovfaces,
author={Postnikov, A. and Reiner, V. and Williams, L.},
title={Faces of generalized permutohedra},
journal={Documenta mathematica},
volume = {13},
pages = {207-273},
year={2008},
langid = {german}}

@article{pretzel,
author={Pretzel, O.},
  title={On reorienting graphs by pushing down maximal vertices},
  journal={Order},
  volume={3},
  pages={135--153},
  year={1986},
  publisher={Springer}}

@article{mosesian,
    AUTHOR = {Mosesian, K. M.},
     TITLE = {Strongly basable graphs},
   JOURNAL = {Akad. Nauk Armjan. SSR Dokl.},
  FJOURNAL = {Akademija Nauk Armjansko\u{\i} SSR. Doklady},
    VOLUME = {54},
      YEAR = {1972},
     PAGES = {134--138},
   MRCLASS = {05C20},
  MRNUMBER = {327560},
MRREVIEWER = {J.\ Bos\'{a}k},}

@article{chen,
  author={Chen, B.},
  title={Orientations, lattice polytopes, and group arrangements I: Chromatic and tension polynomials of graphs},
  journal={Annals of combinatorics},
  volume={13},
  pages={425--452},
  year={2010},
  publisher={Springer}}

@article{eriksson,
author={Eriksson, H. and Eriksson, K.},
  title={Conjugacy of Coxeter elements},
  journal={The Electronic Journal of Combinatorics},
volume = {16},
year={2009}}

@article{macauley2009posets,
  title={Posets from admissible Coxeter sequences},
  author={Macauley, Matthew and Mortveit, Henning S},
  journal={The Electronic Journal of Combinatorics},
  volume = {18},
  year={2011}}

@article{speyer,
author={Speyer, D.},
  title={Powers of Coxeter elements in infinite groups are reduced},
  journal={Proceedings of the American Mathematical Society},
  volume={137},
  number={4},
  pages={1295--1302},
  year={2009}}

@article{bernstein,
  author={Bernstein, I.N. and Gel'fand, I. and Ponomarev, V.A.},
  title={Coxeter functors and Gabriel's theorem},
  journal={Russian mathematical surveys},
  volume={28},
  year={1973},
  number={2}}

@article{Kleiss,
author = {Kleiss, R. and Kuijf, H.},
title = {Multigluon cross sections and 5-jet production at hadron colliders},
journal = {Nuclear Physics B},
volume = {312},
number = {3},
pages = {616-644},
year = {1989}}

@article{Shuffle,
 author = {Eilenberg, S. and Mac Lane, S.},
 journal = {Annals of Mathematics},
 number = {1},
 pages = {55--106},
 title = {On the Groups H($\Pi$, n), I},
 volume = {58},
 year = {1953}}

@article{parketaylor,
author={Parke, S. J. and Taylor, T. R.},
  title={Amplitude for n-gluon scattering},
  journal={Physical Review Letters},
  volume={56},
  number={23},
  year={1986},
  publisher={APS}}

@article{macauley2016morphisms,
author={Macauley, M.},
  title={Morphisms and order ideals of toric posets},
  journal={Mathematics},
  volume={4},
  number={2},
  pages={39},
  year={2016}}

@article{megiddo1976partial,
 author={Megiddo, N.},
  title={Partial and complete cyclic orders},
  journal ={Bull. Amer. Math. Soc.},
  volume = {82},
  number = {3},
pages = {274-276},
  year={1976}}

@misc{Williams,
author={Parisi, M. and Sherman-Bennett, M. and Tessler, R. and Williams, L.},
eprint = {2404.03026},
eprinttype = {arxiv},
title={The Magic Number Conjecture for the $m = 2$ amplituhedron and Parke-Taylor identities},
year={2024}}

@article{fomin2,
author={Fomin, S. and Zelevinsky, A.},
  title={Cluster algebras II: Finite type classification},
 journal={Inventiones mathematicae},
volume = {154},
 pages = {63-121},
year={2002}}

@article{greenekleitman,
author={Greene, C. and Kleitman, D. J.},
title={The structure of Sperner k-families},
journal={Journal of Combinatorial Theory, Series A},
volume={20},
year={1976},
number={1},
pages={41--68}}

@article{brightwell,
  author={Brightwell, G. and Winkler, P.},
  title={Counting linear extensions is \# P-complete},
  journal = {Symposium on the Theory of Computing},
  pages={175--181},
  year={1991}}

@article{pretzel2,
  author={Pretzel, O.},
  title={On reorienting graphs by pushing down maximal vertices—II},
  journal={Discrete mathematics},
  volume={270},
  year={2003},
  number={3},
  pages={227--240}}

@misc{defant,
author={Defant, C. and Kravitz, N.},
eprint ={2009.05040},
eprinttype = {arxiv},
title={Friends and strangers walking on graphs},
year={2020}}

@inproceedings{caldero,
  title={From triangulated categories to cluster algebras II},
  author={Caldero, Philippe and Keller, Bernhard},
  booktitle={Annales scientifiques de l'Ecole normale sup{\'e}rieure},
  volume={39},
  number={6},
  pages={983--1009},
  year={2006}}

@article{coleman,
  author={Coleman, A.J.},
  title={Killing and the Coxeter transformation of Kac-Moody algebras},
  journal={Inventiones mathematicae},
  volume={95},
  year={1989},
  number={3},
  pages={447--477}}

@article{goldreich2008computational,
  title={Computational complexity: a conceptual perspective},
  author={Goldreich, Oded},
  journal={ACM Sigact News},
  volume={39},
  number={3},
  pages={35--39},
  year={2008},
  publisher={ACM New York, NY, USA}
}

@article{williams2014cluster,
  title={Cluster algebras: an introduction},
  author={Williams, Lauren},
  journal={Bulletin of the American Mathematical Society},
  volume={51},
  number={1},
  pages={1--26},
  year={2014}
}

@misc{fomin2017introduction,
  title={Introduction to cluster algebras. chapters 4-5},
  author={Fomin, Sergey and Williams, Lauren and Zelevinsky, Andrei},
  eprint={1707.07190},
  eprinttype = {arxiv},
  year={2017}
}

@book{arora,
  title={Computational complexity: a modern approach},
  author={Arora, Sanjeev and Barak, Boaz},
  year={2009},
  publisher={Cambridge University Press}
}
\end{document}